\documentclass{article}

\usepackage{amsmath}
\usepackage{amsfonts}
\usepackage{amssymb}
\usepackage{amsthm,color}
\usepackage{amsmath,amsthm,amssymb}
\usepackage{graphicx,mathtools}
\usepackage{cite}
\usepackage{enumitem}
\usepackage{lipsum}
\usepackage{algorithm}      
\usepackage{algorithmic}    
\usepackage{dsfont}
\usepackage{hyperref}

\newtheorem{Theorem}{Theorem}[section]

\def\R{{\mathbb R}}

\def\({\left(}
\def\){\right)}

\usepackage{subcaption} 

\graphicspath{{pics/}}

\begin{document}


\begin{flushleft}
\Large 
\noindent{\bf \Large Novel implementation of the extended sampling method for inverse biharmonic scattering}
\end{flushleft}

\vspace{0.2in}

{\bf  \large Isaac Harris }\\
\indent {\small Department of Mathematics, Purdue University, West Lafayette, IN 47907, USA, }\\
\indent {\small Email: \texttt{harri814@purdue.edu} }\\

{\bf  \large General Ozochiawaeze}\\
\indent {\small Department of Mathematics, Purdue University, West Lafayette, IN 47907, USA, } \\
\indent {\small Email: \texttt{oozochia@purdue.edu}}\\

{\bf AMS subject classifications:} {35R30, 35R60}\\

{\bf Keywords:} {inverse flexural wave scattering, biharmonic wave equation, extended sampling method, clamped obstacle, single incident wave}

\begin{abstract}
\noindent This paper considers an inverse shape problem for recovering an unknown clamped obstacle in two dimensions from far--field measurements generated by a single incident wave or just a few incident waves for the biharmonic (flexural) wave equation. Here we will develop a new extended sampling method (ESM) that is derived using the analysis of the well--known factorization method. We will also consider an ESM using both sound--soft and sound--hard sampling disks to identify sampling points where the reference disk intersects the unknown cavity. The use of a  sound--hard sampling disk has not been studied in the literature whereas the  sound--soft sampling disk has been used in most recent works. Traditionally the ESM seeks to find the location of the scatterer from limited incident directional data. Here, our method acts more like the factorization method to obtain the location as well as the size (and possibly the shape) of the obstacle. We present numerical experiments with synthetic data that demonstrate how effective this new implementation is with respect to noisy data and illustrate the influence of the reference disk radius on the reconstruction.
\end{abstract}

\section{Introduction}
In this paper, we consider an inverse shape problem for biharmonic (flexural) waves in a thin elastic plate. Here, we wish to  determine an unknown clamped obstacle $D\subset\R^2$ from measured far--field data generated by one or a few incident plane waves. Inverse flexural wave scattering problems arise in a wide range of applications including non--destructive testing, structural health monitoring, and engineering platonic crystals \cite{MEMMOLO2018568,gao2018theoretical,sozio2023optimal}. Our scattering problem is governed by the Kirchhoff--Love model for a clamped region. The out-of-plane displacement field satisfies a biharmonic Helmholtz equation along with clamped boundary conditions which leads to a fourth--order scattering problem. For this problem, the well--posedness for the direct scattering problem has been studied in \cite{bourgeois2020well} as well as \cite{DongHeping2024ANBI} where a numerical method was also given.

A large class of non--iterative reconstruction techniques for inverse scattering are referred to as {\it qualitative methods}, such as the linear sampling and the factorization methods which have been studied for a wide range of problems. These approaches aim to recover the support of the scatterer by testing whether certain functions belong to the range of a data operator, rather than solving a nonlinear optimization problem. They are typically inexpensive to implement and require little a priori information. One of the main disadvantages is that these methods often require multi--static data collected all around the scatterer. In the context of biharmonic (flexural) wave scattering, qualitative methods have only been developed recently, including linear sampling, direct sampling and  factorization methods see, for e.g., \cite{RafaEtal-simplesupport,bourgeois2020linear,GuoLongWu2024,HK-simplesupport,HLP-directsample,HarrisLiOzochiawaeze2026}.

The focus of our work in this manuscript is the {extended sampling method} (ESM) introduced for inverse acoustic scattering in \cite{LiuSun2018} and later extended to other wave models such as inverse electromagnetic scattering in  \cite{LiuYangYanSun2026}. The ESM is designed to recover the location of an unknown target using far--field measurements generated by one or a small number of incident directions. Conceptually, the method probes the data with a family of {translated sampling domains} (typically disks) and produces an imaging functional whose behavior changes depending on whether the sampling domain intersects the unknown scatterer. This alleviates one of the main drawbacks of many qualitative methods, namely the requirement of full--aperture multi--static measurements.

While the original ESM analysis is closely tied to linear sampling type arguments, the factorization method provides a more analytically rigorous framework based on range characterizations of compact normal operators \cite{Kirsch1998,kirsch2008factorization}. Motivated by this, in this paper we develop a {\it factorization--based} implementation of the ESM for the inverse biharmonic scattering problem with a clamped obstacle. Our approach yields explicit indicators derived from the spectral data of far--field operators associated with translated sampling disks and the measured far--field data. In particular, we construct imaging functionals that can be evaluated given the biharmonic far--field data that can be used to recover the obstacle $D$ by sweeping through sampling points $z \in \R^2$.

A notable feature of our formulation is that we consider two types of sampling disks with {sound--soft} (Dirichlet) and {sound--hard} (Neumann) boundary conditions. The sound--soft variant is the standard choice in most ESM implementations, whereas the sound--hard sampling disk has not been explored in the ESM literature to the best of our knowledge. The resulting pair of indicators provides additional flexibility in practice and offers a second, closely related reconstruction functional derived from the same factorization--based perspective.

In addition to localization, our numerical experiments suggest that the dependence of the indicator on the sampling radius contains information about the size of the obstacle, enabling a simple disk--type approximation that can be used as a coarse geometric proxy for $D$. Such coarse reconstructions can be valuable as an initialization for more refined iterative methods, and can also be incorporated into stochastic inversion frameworks for limited data settings (see, for e.g., \cite{LiDengSun2020,LiSunXu2020,HuangLiXu2026}).

The remainder of the paper is organized as follows. In Section \ref{dp-FMbasedESM} we recall the direct scattering problem for a clamped obstacle and review the operator splitting that reduces the biharmonic model to coupled Helmholtz and modified Helmholtz components. In Section \ref{ip-FMbasedESM} we formulate the inverse problem and derive factorization--based ESM imaging functionals using translated sampling disks, including both Dirichlet and Neumann sampling operators. Section \ref{numerics-FMbasedESM} presents the numerical implementation and reconstruction results for one and multiple incident directions, including tests under noise and examples illustrating the effect of the sampling radius. We conclude in Section \ref{end-FMbasedESM} with a summary and possible extensions.

\section{Direct Scattering Problem}\label{dp-FMbasedESM}
In this section, we will discuss the direct scattering problem associated with our inverse shape problem that we wish to solve using our new ESM. The direct scattering problem comes from  biharmonic (flexural) wave scattering with application to many areas such as non--destructive testing as well as other areas of scientific interest. To this end, we consider a two--dimensional infinite elastic thin plate with an impenetrable clamped obstacle, which is a bounded open domain $D \subset \R^2$ with smooth boundary $\Gamma$. The obstacle is illuminated by a time--harmonic plane wave expressed as
\[
u^i(x,d)=\mathrm e^{\mathrm i\kappa x\cdot d},\quad x\in \R^2,
\]
where $d=(\cos{\theta},\sin{\theta})$ represents the direction of the incident wave for a fixed $\theta\in [0,2\pi)$ which denotes the incident angle, and $\kappa>0$ is the wavenumber. 
The vertical out-of-plane displacement $u$, also known as the total field, satisfies the two--dimensional biharmonic wave equation in the frequency domain
\begin{align}\label{eqn:pde}
    \Delta^2u-\kappa^4u=0,\quad x\in\R^2\setminus\overline D.
\end{align}
On the boundary $\Gamma$, we assume that the total field satisfies the corresponding clamped boundary conditions
\begin{align}\label{eqn:bcs}
    u=0,\quad \partial_\nu u=0,\quad x\in \Gamma,
\end{align}
where $\nu$ is the unit outward normal vector on the boundary $\Gamma$.

Similar to the case of acoustic or electromagnetic scattering, we have that the total field can be expressed as the sum of the incident wave and scattered field
\[
u=u^i+u^s.
\]
Since we have that the incident plane wave satisfies the biharmonic wave equation in the frequency domain by \eqref{eqn:pde} we have that the scattered field $u^s\in H_{\text{loc}}^2(\mathbb R^2\setminus\overline D)$ also satisfies the biharmonic wave equation 
\begin{align}\label{eqn:scat1}
    \Delta^2u^s-\kappa^4u^s=0,\quad x\in\mathbb R^2\setminus\overline D,
\end{align}
and the clamped boundary conditions \eqref{eqn:bcs} become non--homogeneous boundary conditions on $\Gamma$ for the scattered field such that 
\begin{align}\label{eqn:scat2}
    u^s=-u^i,\quad \partial_\nu u^s=-\partial_\nu u^i,\quad x\in\Gamma.
\end{align}
In addition, to close the system we assume that the scattered field and its Laplacian satisfy the Sommerfeld radiation condition
\begin{align}\label{eqn:src}
    \lim_{r\to\infty} \sqrt{r}(\partial_r u^s-\mathrm{i}\kappa u^s)=0,\quad \lim_{r\to\infty}\sqrt{r}(\partial_r \Delta u^s-\mathrm{i} \kappa \Delta u^s)=0,\quad r=|x|.
\end{align}
From \cite{bourgeois2020well,DongHeping2024ANBI} we have that the scattering problem \eqref{eqn:scat1}--\eqref{eqn:src} is well--posed. The radiation conditions say that the scattered field is outgoing. Following \cite{bourgeois2020well}, we define the auxiliary functions
\begin{align*} 
    u^s_{\mathrm H}= -\frac{1}{2\kappa^2}(\Delta u^s-\kappa^2 u^s),\quad u^s_{\mathrm M}= \frac{1}{2\kappa^2}(\Delta u^s+\kappa^2u^s).
\end{align*}
From this, the following relations hold
\begin{align*} 
    u^s=u^s_{\mathrm H}+u^s_{\mathrm M},\quad \Delta u^s=-\kappa^2(u^s_{\mathrm H}-u^s_{\mathrm M}).
\end{align*}
It follows that $u^s_{\text{H}}$ and $u^s_{\text{M}}$ satisfy the Helmholtz and modified Helmholtz equation, respectively,
\begin{align}\label{eqn2:eqns}
    \Delta u^s_{\mathrm H}+\kappa^2 u^s_{\mathrm H}=0,\quad \Delta u^s_{\mathrm M}-\kappa^2 u^s_{\mathrm M}=0,\quad x\in \mathbb R^2\setminus\overline D.
\end{align}
On the boundary $\Gamma$, this would imply that $u^s_{\mathrm H}$ and $u^s_{\mathrm M}$ satisfy the following coupled boundary conditions 
\begin{align}\label{eqn2:coupledbcs}
    u^s_{\mathrm H}+u^s_{\mathrm M}=-u^i,\quad \partial_\nu u^s_{\mathrm H}+\partial_\nu u^s_{\mathrm M}=-\partial_\nu u^i.
\end{align}
Furthermore, $u^s_{\mathrm H}$ and $u^s_{\mathrm M}$ satisfy the following radiation conditions
\begin{align}\label{eqn2:srcs}
    \lim_{r\to\infty} \sqrt{r} (\partial_r u^s_{\mathrm H}-\mathrm{i} \kappa u^s_{\mathrm M})=0,\quad \lim_{r\to\infty}\sqrt{r}(\partial_r u^s_{\mathrm M}-\mathrm{i} \kappa u^s_{\mathrm M})=0,\quad r=|x|.
\end{align}
In summary, the fourth--order biharmonic clamped scattering problem (\ref{eqn:scat1})--(\ref{eqn:src}) is well--posed for any incident direction and is equivalent to the coupled second--order boundary value problem (\ref{eqn2:eqns})--(\ref{eqn2:srcs}).

Now, we recall that the radiating fundamental solution $\Phi_{\text{H}}(x,y)$ of the Helmholtz equation in $\mathbb R^2\setminus\{x\neq y\}$ which is given by
\begin{align}\label{fund_soln_helmholtz}
    \Phi_{\text{H}}(x,y)&=\frac{\mathrm i}{4}H_{0}^{(1)}(\kappa|x-y|)\quad \text{for all}\quad x\neq y,
\end{align}
where $H_{0}^{(1)}$ is the Hankel function of the first kind of order $0$. Similarly, we have that $\Phi_{\text{M}}(x,y)$ is the radiating fundamental solution to the modified Helmholtz equation in $\mathbb R^2\setminus\{x\neq y\}$ which is given by
\begin{align}\label{fund_soln_modhelmholtz}
    \Phi_{\text{M}}(x,y)&=\frac{\mathrm i}{4}H_{0}^{(1)}(\mathrm{i}\kappa|x-y|)\quad \text{for all}\quad x\neq y.
\end{align}
By Green's representation theorem, we have
\[
u_{\mathrm P}^s(x)=\int_{\Gamma}\left(u_{\mathrm P}^s(y)\partial_{\nu(y)}\Phi_{\text{P}}(x,y)-\Phi_{\text{P}}(x,y)\partial_{\nu(y)} u_{\mathrm P}^s(y)  \right)\, \text{d}s(y),\quad x\in \mathbb R^2\setminus\overline D,
\]
for $\text{P}\in \{\text{H},\text{M}\}$.
Since the scattered field $u^s$ is assumed to be radiating, it is known to have the asymptotic behavior
\begin{align}\label{asymptotic_rad}
u^s(x,d)=\frac{\text{e}^{{\rm i}\pi/4}}{\sqrt{8\pi\kappa}} \cdot \frac{\text{e}^{{\rm i}\kappa r}}{\sqrt{r}} \left\{ u^{\infty}(\hat{x},d)+O \left(\frac{1}{r}\right) \right\},\quad \text{as} \,\, \, r \to\infty,
\end{align}
where $u^{\infty}(\hat{x},d)$ is the far--field pattern of $u^s(x,d)$. Notice that we denote the dependence on the incident direction explicitly. For any $\kappa>0$, we see that both 
$$\text{$u^s_{\mathrm M}$ and $\partial_r u^s_{\mathrm M} \,\,\,$ decay exponentially as $\,\,\, r\to\infty$,}$$ 
see, for e.g., \cite{bourgeois2020well}. Just as in the case of studying scattering in a waveguide (see for e.g. \cite{BorceaCMwaveguide2019,Meng-waveguide2021}) we see that the scattered field has a propagation part $u^s_{\text{H}}$ and an evanescent part $u^s_{\text{M}}$. This implies that the far--field pattern for the solution of (\ref{eqn:scat1})--(\ref{eqn:src}) is given by $u^{\infty}=u_{\mathrm H}^{\infty}$, where $u_{\mathrm H}^{\infty}$ is the far--field pattern of $u^s_{\mathrm H}$. From the representation theorem this implies that
\begin{align}\label{green's:ffp}
    u^{\infty}(\hat{x},d)&=\int_{\Gamma}\left(u^s_{\mathrm H}(y,d)\partial_{\nu(y)}\mathrm{e}^{-\text{i}\kappa \hat{x}\cdot y}-\partial_{\nu(y)}u^s_{\mathrm H}(y,d)\mathrm{e}^{-\text{i}\kappa \hat{x}\cdot y}\right)\,\text{d}s(y),
\end{align}
where $\mathrm e^{-\mathrm i\kappa\hat{x}\cdot y}$ is the far--field pattern of $\Phi_{\text{H}}(x,y)$.

This paper is concerned with the inverse scattering problem of recovering the location and approximation of the support of the clamped obstacle $D$ from the knowledge of the far--field pattern $u^{\infty}(\hat{x} , d)$ for all observation directions $\hat{x} \in \mathbb{S}^1$ and a finite number of incident directions $d \in \mathbb{S}^1$, where $\mathbb{S}^1$ denotes the unit circle in $\R^2$. An interesting observation is that, even though the far--field data does not contain any information from $u^s_{\mathrm M}$ directly, it is still able to uniquely recover the clamped cavity $D$ as shown in \cite{Dong2023UniquenessOA}.

\section{Characterization of the Scatterer via ESM}\label{ip-FMbasedESM}
This section is devoted to the analysis of a novel implementation of the extended sampling method (ESM) which was first introduced in \cite{LiuSun2018} for inverse acoustic scattering problems.  Our aim is to improve upon the qualitative reconstruction of clamped obstacles relative to the classical ESM which was studied in \cite{HarrisLiOzochiawaeze2026}. In particular, we aim to develop a method that can accurately recover the location and size of the scatterer from a finite number of incident directions. 

In order to proceed, we first prove an analogous result to Theorem 2.1 in \cite{LiuYangYanSun2026}, i.e., any two bounded obstacles that generate identical non--trivial far--field patterns for the radiation solutions to the Helmholtz equation in their respective exteriors must intersect. This is the main idea we will use to develop a new ESM applied to our inverse biharmonic scattering problem. 

\begin{Theorem}\label{weak_uniqueness_thm}
Let $D_1, D_2 \subset \mathbb{R}^2$ be bounded obstacles, and let $v_1^s$ and $v_2^s$ be radiating solutions to the Helmholtz equation in $\mathbb{R}^2 \setminus D_1$ and $\mathbb{R}^2 \setminus D_2$, respectively. If the non--trivial far--field patterns $v_1^\infty$ and $v_2^\infty$ satisfy that 
$$v_1^\infty(\hat{x}) = v_2^\infty(\hat{x}), \quad \text{for all} \,\,\,\,  \hat{x} \in \mathbb{S}^1, \quad \text{then} \quad \overline D_1 \cap \overline D_2 \neq \emptyset.$$
\end{Theorem}
\begin{proof}
To prove this result, we argue by way of contradiction. Therefore, we will assume that the bounded obstacles are disjoint such that 
\[
\overline D_1 \cap \overline D_2 = \emptyset. 
\]
Since the corresponding far--field patterns are equal, i.e., $v_1^\infty = v_2^\infty$ Rellich's lemma (see, for e.g., Lemma 2.1in  \cite{kirsch2008factorization}) and the unique continuation principle implies that 
\[ v_1^s (x) = v_2^s(x) \quad \text{for all} \,\,\,\,  x \in \mathbb{R}^2 \setminus (\overline{D_1 \cup D_2}). \]
Now, we define the function
\[
v(x) =
\begin{cases}
v_1^s(x), & x \in \mathbb{R}^2 \setminus \overline{D_1}, \\
v_2^s(x), & x \in D_1.
\end{cases}
\]
This is a well--defined smooth function since the two scattered fields coincide in $\mathbb{R}^2 \setminus (\overline{D_1 \cup D_2})$. Moreover, $v$ satisfies the Helmholtz equation
\[
(\Delta + \kappa^2)v = 0 \quad \text{in all of } \,\,\, \mathbb{R}^2,
\]
together with the Sommerfeld radiation condition. Since $v$ is a radiating solution to the Helmholtz equation in all of $\R^2$ we conclude that 
\[
v(x) \equiv 0 \quad \text{for all} \,\,\,\,  x \in \mathbb{R}^2.
\]
In particular, the far--field pattern vanishes, i.e.,
\[
v^\infty = v_1^\infty = v_2^\infty = 0,
\]
which contradicts the assumption that the far--field patterns are non--trivial. Therefore, we obtain that 
\[
D_1 \cap D_2 \neq \emptyset,
\]
as claimed.
\end{proof}

With this result we can revisit the ESM for recovering our clamped obstacle $D$. Even though the unknown obstacle is a clamped region and generates a biharmonic scattered field $u^s$, the far--field pattern is determined only by its Helmholtz component $u_{\mathrm{H}}^s$. This would imply that even though Theorem \ref{weak_uniqueness_thm} is stated for acoustic scatterers, this result can be applied to our biharmonic far--field data.

\subsection{Sound--Soft ESM }
We now revisit the ESM for the inverse biharmonic scattering problem of reconstructing a clamped obstacle from far--field data corresponding to one or finitely many incident directions. Specifically, we aim to determine both the location and an approximation of the size of the region $D$ from the known/measured far--field pattern $u^{\infty}(\hat{x},d)$, for all $(\hat{x},d) \in \mathbb{S}^1\times \mathbb{S}_{\mathrm{inc}}^1$ for a fixed wavenumber $\kappa >0$. Here, in our notation 
$$\mathbb{S}_{\mathrm{inc}}^1\subsetneq \mathbb{S}^1 \quad \text{such that} \quad \mathbb{S}_{\mathrm{inc}}^1 = \{ d_1, d_2, \dots, d_{N_{\mathrm{inc}}} \}$$ 
for some fixed $N_{\mathrm{inc}} \in \mathbb{N}$. The finite collection of incident directions is assumed to satisfy $d_\ell \neq d_{\ell'}$ for all $\ell \neq \ell'$. For instance, if $\mathbb{S}_{\mathrm{inc}}^1 = \{d_1 \}$, the data correspond to a single incident wave. In contrast, for the case of multi--static far--field data one assumes that $u^{\infty}(\hat x,d)$ is known for all $\hat{x} , d \in \mathbb S^1$, i.e., $\mathbb{S}_{\mathrm{inc}}^1=\mathbb S^1$. When one has access to the multi--static far--field data, then the linear sampling or direct sampling methods studied in \cite{HLP-directsample,HarrisLiOzochiawaeze2026} can be used to recover the clamped obstacle.

Throughout the paper, the adjoint of an operator will be denoted by $(\cdot)^{*}$ and the inner product over $L^2(\mathbb S^1)$ by $\langle \cdot \, ,\cdot\rangle_{L^2(\mathbb S^1)}$. We also let $B_0 \subset \mathbb{R}^2$ denote a reference disk centered at the origin, of a fixed radius $R>0$. For a given sampling point $z \in \mathbb{R}^2$, we define the shifted sampling disk 
$$B_z := \{\, x + z \, : \, x \in B_0 \,\}. $$
With this, we let $U^s_{\mathrm{Dir} , B_z}(x,\hat y)$ denote the scattered field in the exterior of $B_z$ corresponding to an incident plane wave incident from direction $\hat y \in \mathbb{S}^1$ with wavenumber $\kappa>0$ that illuminates the sound--soft (i.e. Dirichlet) sampling disk $B_z$. Then $U^s_{\mathrm{Dir} , B_z} \in H_{\mathrm{loc}}^1(\mathbb{R}^2 \setminus \overline{B_z})$ satisfies the exterior boundary value problem
\begin{align}\label{ref_disk_bvp}
\begin{dcases}
\Delta U^s_{\mathrm{Dir} , B_z} + \kappa^2 U^s_{\mathrm{Dir} , B_z} = 0, \quad \text{in } \mathbb{R}^2 \setminus \overline{B_z},\\[1mm]
U^s_{\mathrm{Dir} , B_z} = -  \mathrm{e}^{\mathrm{i} \kappa x\cdot \hat y} , \quad \text{on } \partial B_z,\\[1mm]
\displaystyle \lim_{r \to \infty} \sqrt{r} \left( \partial_r U^s_{\mathrm{Dir} , B_z} - \mathrm i \kappa U^s_{\mathrm{Dir} , B_z} \right) = 0.
\end{dcases}
\end{align}
By separation of variables (see, for e.g., \cite{LiuSun2018}) applied to a sound--soft disk centered at the origin, the far--field pattern of $U^s_{\mathrm{Dir} , B_0}(x,\hat y)$ is given by
\begin{align}\label{sep_var}
U_{\mathrm{Dir} , B_0}^{\infty}(\hat x,\hat y)= -\frac{4}{\text{i}}\left[ \frac{J_0(\kappa R)}{H_0^{(1)}(\kappa R)} + 2\sum_{n=1}^{\infty}\frac{J_n(\kappa R)}{H_n^{(1)}(\kappa R)} \cos\big(n(\theta_x-\theta_y)\big) \right],
\,\,\, \text{ for all} \,\,\, (\hat{x},\hat{y}) \in \mathbb{S}^1\times \mathbb{S}^1,
\end{align}
where $J_n$ and $H_n^{(1)}$ denote the Bessel and Hankel functions of the first kind of order $n$, respectively. Also, in the above series expansion $\theta_x$, $\theta_y$ are the observation and incident angles corresponding to $\hat x$ and $\hat y$.
For a sampling disk centered at $z \in \mathbb{R}^2$, the far--field pattern satisfies the translation relation
\begin{align}\label{trans_prop}
U_{\mathrm{Dir} , B_z}^{\infty}(\hat x,\hat y)
= \mathrm e^{-\mathrm{i}\kappa z\cdot(\hat{x}-\hat{y})} U_{\mathrm{Dir} , B_0}^{\infty}(\hat x,\hat y),
\,\,\, \text{ for all} \,\,\, (\hat{x},\hat{y}) \in \mathbb{S}^1\times \mathbb{S}^1.
\end{align}
One thing to note is that $U_{\mathrm{Dir} , B_z}^{\infty}(\hat x,\hat y)$, obtained from \eqref{sep_var} and \eqref{trans_prop}, is independent of the unknown clamped obstacle $D$. Therefore, we can use this precomputed far--field pattern together with the measured data $u^{\infty}(\hat x,d)$ to derive a qualitative method for recovering $D$.

Let $\mathcal{M}$ denote a sampling domain containing $D$. For the ESM studied in \cite{HarrisLiOzochiawaeze2026} at each sampling point $z \in \mathcal{M}$ and a fixed incident direction $d \in \mathbb{S}^1$, one seeks $g_z \in L^2(\mathbb{S}^1)$ that satisfies the far--field equation 
\begin{align}\label{ESM_Eqn1}
(\mathcal F_{B_z}^{\mathrm{Dir}} g_z)(\hat x) = u^{\infty}(\hat x ,d ), \,\,\, \text{ for all} \,\,\, \hat{x}  \in \mathbb{S}^1,
\end{align}
where the Dirichlet far--field operator associated with the sampling disk $B_z$ is denoted 
\begin{align}\label{ESM_FF_op}
\mathcal F_{B_z}^{\mathrm{Dir}}: L^2(\mathbb{S}^1) \to L^2(\mathbb{S}^1) \quad \text{and is defined by} \quad (\mathcal F_{B_z}^{\mathrm{Dir}} g)(\hat x)
= \int_{\mathbb{S}^1} U_{\mathrm{Dir} , B_z}^{\infty}(\hat x,\hat y)\, g(\hat y)\, \text{d}s(\hat y).
\end{align}
It was shown in \cite{HarrisLiOzochiawaeze2026} that, using the Dirichlet far--field operator, one obtains that roughly speaking, if $D \subset B_z$, then \eqref{ESM_Eqn1} is solvable via any regularization strategy such that the solution $g_z$ has a bounded norm. In contrast, if $D \cap B_z = \emptyset$, then any approximate solution of \eqref{ESM_Eqn1} has an unbounded norm as the regularization parameter tends to zero. Note that this implementation may require {\it a priori} knowledge of the size of $D$ to pick an appropriate radius for the reference disk $B_z$.

We now take a different perspective on the ESM motivated by the result given in Theorem \ref{weak_uniqueness_thm}. This result shows that identical far--field patterns cannot arise from two disjoint obstacles. In particular, it suggests that the compatibility of the measured far--field data $u^\infty$ with that of a sampling disk $B_z$ is inherently linked to whether $B_z$ intersects the true obstacle $D$. This observation provides a framework for characterizing admissible sampling points $z \in \mathcal{M}$ that lie in or near the true obstacle. To determine admissible sampling points our method will use the analytical framework of the factorization method described in \cite{Kirsch1998,kirsch2008factorization}. The previous ESM arguments (see, for e.g. \cite{HuangLiXu2026,LiDengSun2020,LiuSun2018,LiuYangYanSun2026}) use the analytical framework of the linear sampling method. 

With this in mind, we now consider an alternative formulation of the ESM based on the $(F^{*}F)^{1/4}$--factorization method. It is well known that the Dirichlet far--field operator $\mathcal F_{B_z}^{\mathrm{Dir}}$ is compact and normal. For each sound--soft sampling disk $B_z$, it admits the factorization 
\begin{align}\label{factorization}
\mathcal F_{B_z}^{\mathrm{Dir}}
= -\mathcal G_{B_z}^{\mathrm{Dir}} \mathcal S_{B_z}^{*} \mathcal G_{B_z}^{\mathrm{Dir}*}.
\end{align}
Here the data-to-pattern operator 
\begin{align}
\mathcal G_{B_z}^{\mathrm{Dir}}: H^{1/2}(\partial B_z) \to L^2(\mathbb{S}^1) \quad \text{ is defined by} \quad \mathcal G_{B_z} f = V_{B_z}^{\infty},
\end{align}
where $V_{B_z}^{\infty} \in L^2(\mathbb{S}^1)$ is the far--field pattern of the radiating solution $V_{B_z}$ to the exterior scattering problem (\ref{ref_disk_bvp}) with the boundary datum $V_{B_z}=f$ on $\partial B_z$ such that  $f \in H^{1/2}(\partial B_z)$. In addition, $\mathcal S_{B_z}^{*}$ is the adjoint of the single--layer potential operator 
\begin{equation*}
 \mathcal S_{B_z}: H^{-1/2}(\partial B_z)\to H^{1/2}(\partial B_z)\quad \text{ defined by} \quad (\mathcal S_{B_z}\psi)(x)= \int_{\partial B_z}\Phi_{\mathrm H}(x,y)\psi(y)\,\text{d}s(y) \big|_{x\in \partial B_z},
\end{equation*}
where again $\Phi_{\mathrm H}(x,y)$ is the radiating fundamental solution of the Helmholtz equation given by (\ref{fund_soln_helmholtz}). Therefore, the factorization method considers a modified version of equation \eqref{ESM_Eqn1} such that for a fixed incident direction $d \in \mathbb{S}^1$, one seeks $g_z \in L^2(\mathbb{S}^1)$ that satisfies the modified far--field equation
\begin{align}\label{ESM_mod}
    \big((\mathcal F_{B_z}^{\mathrm{Dir}*}\mathcal F_{B_z}^{\mathrm{Dir}})^{1/4} g_z\big)(\hat{x})
    = u^\infty(\hat{x} ,d), \,\,\, \text{ for all} \,\,\, \hat{x} \in \mathbb S^1.
\end{align}
The analysis of the $(F^{*}F)^{1/4}$--factorization method, Theorem 1.24 in \cite{kirsch2008factorization} yields the range identity  
\[
\text{Range}(\mathcal G_{B_z}^{\mathrm{Dir}})=\text{Range}\left( (\mathcal F_{B_z}^{\mathrm{Dir}*}\mathcal F_{B_z}^{\mathrm{Dir}})^{1/4} \right)
\]
provided that $\kappa^2$ is not a Dirichlet eigenvalue of the negative Laplacian in $B_z$. Hence, by the Picard range theorem together with Theorem \ref{weak_uniqueness_thm}, we obtain the following result.

\begin{Theorem}\label{refined_ESM_thm}
Assume that $\kappa^2$ is not a Dirichlet eigenvalue of the negative Laplacian in the sound--soft reference disk $B_z$. If we let the orthonormal eigensystem of the Dirichlet far--field operator $\mathcal{F}_{B_z}^{\mathrm{Dir}}$ be denoted  
$$\big(\lambda_{z}^{(j)}, \varphi_{z}^{(j)} \big) \in \mathbb{C} \setminus \{0\} \times L^2(\mathbb{S}^1).$$ 
Then the indicator function
\[
W(z) = \left[ \sum_{j=1}^\infty \frac{\big|\langle u^{\infty}, \varphi_{z}^{(j)} \rangle_{L^2(\mathbb{S}^1)}\big|^2}{|\lambda_{z}^{(j)}|} \right]^{-1} > 0 
\quad \implies \quad \overline B_z \cap \overline D \neq \emptyset.
\]
\end{Theorem}
\begin{proof}
    For the sound--soft disk $B_z$, we have the factorization (\ref{factorization}), i.e., 
    \[\mathcal F_{B_z}^{\mathrm{Dir}}=-\mathcal G_{B_z}^{\mathrm{Dir}}\mathcal S_{B_z}^{*}\mathcal G_{B_z}^{\mathrm{Dir}*},\] 
    where $\mathcal G_{B_z}^{\mathrm{Dir}}: H^{1/2}(\partial B_z)\to L^2(\mathbb S^1)$ is the data-to-pattern operator corresponding to $B_z$. By Picard's range theorem applied to the factorization method range identity \cite{kirsch2008factorization}, we see that 
    $$\text{ $W(z)>0\,\,$ if and only if $\,\,u^{\infty}\in \text{Range} \left((\mathcal F_{B_z}^{\mathrm{Dir}*}\mathcal F_{B_z}^{ \mathrm{Dir} })^{1/4} \right)$}.$$ 
Since we have that the range identity from the factorization method, i.e., 
$$\text{Range}(\mathcal G_{B_z}^{\mathrm{Dir}})=\text{Range} \left((\mathcal F_{B_z}^{\mathrm{Dir}*}\mathcal F_{B_z}^{\mathrm{Dir}} )^{1/4}\right)$$ 
we have that 
$$ \text{$u^{\infty}\in \text{Range}(\mathcal G_{B_z}^{\mathrm{Dir}})\,\,$ if and only if $\,\,W(z)>0$}.$$ 
Recalling the definition of $\mathcal G_{B_z}^{\mathrm{Dir}}$, it follows that there exists $f\in H^{1/2}(\partial B_z)$ such that $u^{\infty}(\hat x)=V_{B_z}^{\infty}(\hat x)$ for all $\hat x\in\mathbb S^1$.  Since $u^{\infty}$ corresponds to the far--field pattern of a solution of the Helmholtz equation in $\R^2 \setminus \overline{D}$ and $V_{B_z}^{\infty}$ is the far--field pattern of a solution of the Helmholtz equation in $\R^2 \setminus \overline{B}_z$ we can appeal to Theorem \ref{weak_uniqueness_thm}, which implies that  
\[
W(z) > 0 \quad \implies \quad \overline B_z \cap \overline D \neq \emptyset,
\]   
completing the proof.
\end{proof}

This result implies that one can use the spectral data for the Dirichlet far--field operator $\mathcal{F}_{B_z}^{\mathrm{Dir}}$ along with the measured far--field data $u^\infty$ corresponding to the clamped obstacle to define a new imaging functional. This is useful in the sense that other qualitative methods need multi--static data whereas the ESM does not. As we will see in our numerical experiments, this method can produce good reconstructions from data produced by only a few incident directions. Also, note that even though we use the sound--soft sampling disk $B_z$ this can be applied to any sampling disk for which the $(F^{*}F)^{1/4}$--factorization method is valid.

\subsection{Sound--Hard ESM}
In this section, we consider a similar ESM where one uses a different sampling disk $B_z$. This is one of the novelties of the ideas presented here, i.e., this method is flexible in the choice of sampling disks. Notice that we mainly rely on Theorem \ref{weak_uniqueness_thm} along with the well-established factorization method to define our imaging functional.  Here we will consider the case when one uses a sound--hard (i.e. Neumann) sampling disk $B_z$ to derive an ESM using the $(F^{*}F)^{1/4}$--factorization method.

To this end, similar to the previous section we let $U^s_{\mathrm{Neu} , B_z}(x,\hat y)$ denote the scattered field in the exterior of $B_z$ corresponding to an incident plane wave with incident direction $\hat y \in \mathbb{S}^1$ and wavenumber $\kappa>0$ that illuminates the sound--hard sampling disk $B_z$. Again, we have that the scattered field $U^s_{\mathrm{Neu} , B_z} \in H_{\mathrm{loc}}^1(\mathbb{R}^2 \setminus \overline{B_z})$ and is the solution to the exterior boundary value problem
\begin{align}\label{ref_disk_bvp2}
\begin{dcases}
\Delta U^s_{\mathrm{Neu} , B_z} + \kappa^2 U^s_{\mathrm{Neu} , B_z} = 0, \quad \text{in } \mathbb{R}^2 \setminus \overline{B_z},\\[1mm]
\partial_\nu U^s_{\mathrm{Neu} , B_z} = -  \partial_{\nu} \mathrm{e}^{\mathrm{i} \kappa x\cdot \hat y} , \quad \text{on } \partial B_z,\\[1mm]
\displaystyle \lim_{r \to \infty} \sqrt{r} \left( \partial_r U^s_{\mathrm{Neu} , B_z} - \mathrm i \kappa U^s_{\mathrm{Neu} , B_z} \right) = 0.
\end{dcases}
\end{align}
Separation of variables implies that the far--field pattern of the reference disk centered at the origin $B_0$ is given by
\begin{align*}
U_{\mathrm{Neu} , B_0}^{\infty}(\hat x,\hat y)= -\frac{4}{\text{i}}\left[ \frac{J_1(\kappa R)}{H_1^{(1)}(\kappa R)} + 2\sum_{n=1}^{\infty}\frac{J_{n-1}(\kappa R) - J_{n+1}(\kappa R)}{H_{n-1}^{(1)}(\kappa R)-H_{n+1}^{(1)}(\kappa R)} \cos\big(n(\theta_x-\theta_y)\big) \right],
\,\,\, \text{ for all} \,\,\, (\hat{x},\hat{y}) \in \mathbb{S}^1\times \mathbb{S}^1.
\end{align*}
Now, by appealing to the fact that the scattering problem \eqref{ref_disk_bvp2} also has the translation relation we obtain that 
\begin{equation}\label{SH_farfield_recursive}
U_{\mathrm{Neu} ,B_z}^{\infty}(\hat{x}, \hat{y}) =  \mathrm{e}^{-\mathrm{i} \kappa z \cdot (\hat{x}-\hat{y})} U_{\mathrm{Neu} , B_0}^{\infty}(\hat x,\hat y),
\,\,\, \text{ for all} \,\,\, (\hat{x},\hat{y}) \in \mathbb{S}^1\times \mathbb{S}^1.
\end{equation}
Again, we have that $J_n$ and $H_n^{(1)}$ are the Bessel and Hankel functions of the first kind of order $n$, respectively, and $\theta_x$, $\theta_y$ are the observation and incident angles corresponding to $\hat x$ and $\hat y$.

Now, to continue deriving another new ESM imaging functional we define the Neumann far--field operator associated with the sampling disk $B_z$, which we denote by
\begin{align}\label{ESM_FF_op2}
\mathcal F_{B_z}^{\mathrm{Neu}}: L^2(\mathbb{S}^1) \to L^2(\mathbb{S}^1) \quad \text{and is defined by 
} \quad (\mathcal F_{B_z}^{\mathrm{Neu}} g)(\hat x) = \int_{\mathbb{S}^1} U_{\mathrm{Neu} , B_z}^{\infty}(\hat x,\hat y)\, g(\hat y)\, \text{d}s(\hat y)
\end{align}
where $U_{\mathrm{Neu} , B_z}^{\infty}(\hat x,\hat y)$ is given by \eqref{SH_farfield_recursive}. Consistent with the sound--soft case, it is well--known that the Neumann far--field operator $\mathcal{F}_{B_z}^{ \mathrm{Neu} }$ is normal and compact. The Neumann far--field operator admits the following factorization derived in Theorem 1.26(a) in \cite{kirsch2008factorization}
\begin{align}\label{SH_factorization}
\mathcal{F}_{B_z}^{\mathrm{Neu}}
= -\, \mathcal{G}_{B_z}^{\mathrm{Neu}} \, 
   \mathcal{N}_{B_z}^{*} \, 
   \mathcal{G}_{B_z}^{\mathrm{Neu}*},
\end{align}
where the hypersingular boundary integral operator is denoted 
\begin{align*}
\mathcal{N}_{B_z}: H^{1/2}(\partial B_z) \to H^{-1/2}(\partial B_z)  \,\,\, \text{and is defined by 
} \,\,\,    (\mathcal{N}_{B_z} \psi)(x) = \frac{\partial}{\partial \nu(x)} \int_{\partial B_z} \frac{\partial \Phi_{\mathrm H}(x, y)}{\partial \nu(y)} \psi(y) \, \text{d}s(y)\big|_{x\in \partial B_z}.
\end{align*}
The data-to-pattern operator for this case is denoted by $\mathcal{G}_{B_z}^{\mathrm{Neu}}: H^{-1/2}(\partial B_z) \to L^2(\mathbb{S}^1)$ and is defined by 
\begin{align}
\mathcal{G}_{B_z}^{\mathrm{Neu}}: H^{-1/2}(\partial B_z) \to L^2(\mathbb{S}^1) \quad \text{and is defined by 
} \quad \mathcal{G}_{B_z}^{\mathrm{Neu}} h = W_{B_z}^{\infty},
\end{align}
where $W_{B_z}^{\infty} \in L^2(\mathbb{S}^1)$ is the far--field pattern of the radiating solution $W_{B_z}$ to the exterior scattering problem (\ref{ref_disk_bvp2}) with the boundary datum $\partial_\nu W_{B_z}=f$ on $\partial B_z$ such that  $f \in H^{-1/2}(\partial B_z)$.

Since we have a symmetric factorization of the operator $\mathcal{F}_{B_z}^{\mathrm{Neu}}$ and by the properties of the hypersingular boundary integral operator we can again appeal to the $(F^{*}F)^{1/4}$--factorization method. Therefore, Theorem 1.26 in \cite{kirsch2008factorization} yields the range identity 
\[
\text{Range}(\mathcal G_{B_z}^{\mathrm{Neu}})=\text{Range} \left( (\mathcal F_{B_z}^{\mathrm{Neu}*}\mathcal F_{B_z}^{\mathrm{Neu}})^{1/4} \right)
\]
if $\kappa^2$ is not a Neumann eigenvalue of the negative Laplacian of $B_z$. Thus, an analogous result to Theorem \ref{refined_ESM_thm} holds for the Neumann far--field operator. 

\begin{Theorem}\label{refined_ESM_thm_neu}
Assume that $\kappa^2$ is not a Neumann eigenvalue of the negative Laplacian in the sound--hard reference disk $B_z$. If we let the orthonormal eigensystem of the Neumann far--field operator $\mathcal{F}_{B_z}^{\mathrm{Neu}}$ be denoted  
$$\big(\lambda_{z}^{(j)}, \varphi_{z}^{(j)} \big) \in \mathbb{C} \setminus \{0\} \times L^2(\mathbb{S}^1).$$ 
Then the imaging functional
\[
W(z) = \left[ \sum_{j=1}^\infty \frac{\big|\langle u^{\infty}, \varphi_{z}^{(j)} \rangle_{L^2(\mathbb{S}^1)}\big|^2}{|\lambda_{z}^{(j)}|} \right]^{-1} > 0 
\quad \implies \quad \overline B_z \cap \overline D \neq \emptyset.
\]
\end{Theorem}

Note that we have omitted the proof for this case as it is identical to the sound--soft case that was presented in the previous section. It is worth noting that, since we have control over the choice of radius $R$ of the sampling disk in both cases, it can be selected so that $\kappa^2>0$ does not coincide with any Dirichlet or Neumann eigenvalues of the negative Laplacian in $B_z$. Thus, the ESM is applicable for all wavenumbers $\kappa$.

{\bf Choosing the sampling radius $R$:} An important step in implementing the new ESM discussed here is the selection of the radius $R$ of the sampling disk $B_z$. We adopt a multi--scale refinement approach. Specifically, starting from an initial radius $R_0>0$, we consider a sequence of radii
\[
R_p = \gamma^p R_0, \quad p \in \mathbb{N} \cup \{0\}
\]
where $\gamma > 1$ is a fixed growth factor. For successive radii $R_p$, the imaging functional $W(z)$ given in Theorem \ref{refined_ESM_thm} and \ref{refined_ESM_thm_neu} can be computed. To obtain the best reconstruction, one can then pick the first radius corresponding to a reconstruction that does not have any `artifacts'. This approach allows for a more refined exploration of the transition regime in which the sampling disk $B_z$ begins to effectively probe the obstacle $D$. In particular, our empirical observations suggest that the reconstruction quality improves significantly when the radius is comparable to the diameter of the scatterer, i.e. $R_p \approx 0.5\cdot \mathrm{diam}(D)$. This results in a more detailed and stable characterization of the obstacle, yielding an improved reconstruction of its location and approximate size.

{\bf Imaging with multiple incident directions:} The reconstruction obtained by applying Theorem \ref{refined_ESM_thm} and \ref{refined_ESM_thm_neu} can be done with measurements from a single incident direction. In many practical scenarios, the far--field measurements are collected at finitely many incident directions,
\begin{align*}
u^{\infty}(\cdot \, ,d_{\ell}), \quad d_{\ell} \in \{ d_1, d_2, \dots, d_{N_{\mathrm{inc}}} \} = \mathbb{S}^1_{\mathrm{inc}} \subseteq \mathbb{S}^1,
\end{align*}
where $N_{\mathrm{inc}}$ is the number of incident directions. For each incident direction $d_{\ell}$, we define the indicator function for the $\ell^{\text{th}}$--incident direction as
\begin{align} \label{single_W}
W(z; d_{\ell}) = 
\left[ \sum_{j=1}^{\infty} 
\frac{\big| \langle u^\infty(\cdot \, , d_{\ell}), \varphi_z^{(j)} \rangle_{L^2(\mathbb{S}^1)} \big|^2}{|\lambda_z^{(j)}|} \right]^{-1},
\end{align}
where $\big(\lambda_{z}^{(j)}, \varphi_{z}^{(j)} \big) \in \mathbb{C} \setminus \{0\} \times L^2(\mathbb{S}^1)$ is the orthonormal eigensystem for the operator $\mathcal{F}_{B_z}^{\sigma}$ for $\sigma = \{$Dir,Neu$\}$.  To aggregate multiple incident directions, we set
\begin{equation}\label{multi_W}
W_{\mathrm{multi}}(z) = \sum_{\ell=1}^{N_{\mathrm{inc}}} W(z; d_\ell).
\end{equation}
Since $\mathcal{F}_{B_z}^{\sigma}$ is compact, its eigenvalues accumulate at zero. This means that as defined, the imaging functional can be unstable with respect to noisy data. In practice, we stabilize the computation using a truncated sum influenced by the spectral cutoff regularization. Here we only retain the eigenvalues satisfying $|\lambda_z^{(j)}| > \alpha$ for a small threshold $\alpha>0$. 

The regularized indicator function for a single incident direction is then
\begin{align}\label{reg_W}
W_\alpha(z; d_{\ell}) \coloneqq 
\left[ \sum_{j=1}^{\infty} \mathds{1}_{[\alpha, \infty)} \big( |\lambda_j| \big)  \frac{\big| \langle u^\infty(\cdot \, ,d_{\ell}), \varphi_z^{(j)} \rangle_{L^2(\mathbb{S}^1)} \big|^2}{|\lambda_z^{(j)}|} \right]^{-1}, \quad \text{where} \quad \displaystyle{   \mathds{1}_{[\alpha, \infty)} \big( t \big) = \left\{\begin{array}{lr} 1, &  t \geq \alpha,  \\
 				&  \\
 0,&  t < \alpha
 \end{array} \right.} 
\end{align}
and for multiple incident directions we use 
\begin{align}\label{reg_multi_W}
W_\alpha^{\mathrm{multi}}(z) = \sum_{\ell=1}^{N_{\text{inc}}} W_\alpha(z; d_\ell).
\end{align}
The key qualitative property derived from Theorems \ref{refined_ESM_thm} and \ref{refined_ESM_thm_neu} is the support of the imaging functional should be able to approximate the obstacle $D$. We will test this idea in our numerical results section.

\section{Numerical Results}\label{numerics-FMbasedESM}
In this section, we will present numerical reconstructions using the regularized imaging functional defined by \eqref{single_W} and \eqref{multi_W}. We demonstrate that the proposed factorization--based ESM implementation is capable of recovering/approximating the scatterer from a small number of incident directions. As previously stated, since the eigenvalues of the Dirichlet or Neumann far--field operator accumulate at zero, one needs to regularize the imaging functional to ensure stability.

The scatterers that we will consider are a star and peanut--shaped region whose boundaries are defined in Table \ref{scatterers}. Here in Figure \ref{fig:scatterers} we plot the scatterers for a visualization. In our experiments we will also consider the scatterers shifted such that 
$$\partial D = x(t) + (s_1,s_2) , \quad \text{for} \quad 0 \leq t < 2\pi$$
where $(s_1,s_2)$ specifies the center of the scatterer and $x(t)$ is given by Table \ref{scatterers}. 
\begin{table}[H]
\centering
\begin{tabular}{ l l }
\hline
\hline
Scatterer & Parameterization \\
\hline
\vspace{-2.0ex}\\
Star--shaped 
& $x(t) = \displaystyle \big(1 + 0.3\cos(4t)\big)\,(\cos t,\sin t)$ 
\vspace{1.5ex} \\ 
Peanut--shaped 
& $x(t) = \displaystyle \left(\tfrac{1}{2}\sqrt{3\cos^2 t + 1}\right)\,(\cos t,\sin t)$  
\vspace{1.5ex}\\ 
\hline
\end{tabular}
\caption{Boundary parameterizations given by $x(t)$ for $t \in [0,2\pi)$.}
\label{scatterers}
\end{table}

\begin{figure}[h]
\centering
\begin{minipage}{0.4\textwidth}
    \centering
    \includegraphics[width=2.5in]{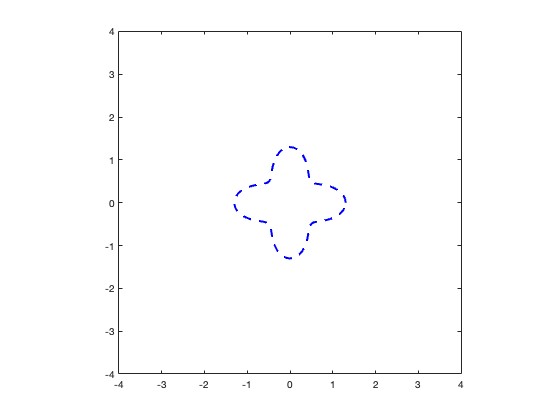}\\
    \small (a) Star--shaped scatterer
\end{minipage}
\begin{minipage}{0.4\textwidth}
    \centering
    \includegraphics[width=2.5in]{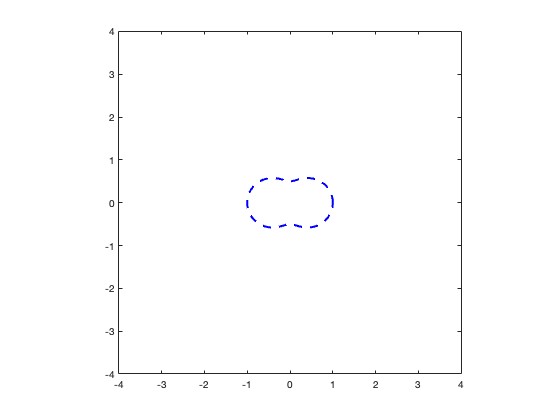}\\
    \small (b) Peanut--shaped scatterer
\end{minipage}
\vspace{-0.5em}
\caption{Visual representations of the two scatterers centered at the origin defined in Table 1.}
\label{fig:scatterers}
\end{figure}

For all of our numerical examples, we use synthetic far--field data corresponding to the two--dimensional biharmonic scattering problem (\ref{eqn:scat1})--(\ref{eqn:src}). Here we use the numerical method developed in \cite{DongHeping2024ANBI} to compute the synthetic data for a clamped obstacle. This method uses a system of boundary integral equations for the equivalent scattering problem given in \eqref{eqn2:eqns}--\eqref{eqn2:srcs}. Using the double and single--layer potential ansatz for the propagating and evanescent fields, it is shown that the system of boundary integral equations is well--posed for a clamped obstacle with smooth boundary. In particular, using this method, the far--field data $u^{\infty}(\hat x,d)$ for $(\hat x,d) \in \mathbb{S}^1 \times \mathbb{S}^1_{\mathrm{inc}}$ is computed. In particular, we define the noisy data vector by
\[
u^\infty_\delta (\hat x_{i} , d ) = u^\infty (\hat x_{i} , d) + \delta \,\frac{\|u^\infty( \cdot \, , d)\|_{2}}{\|\eta\|_{2}}\,\eta(i)
\]
where 
\[
\hat x_{i} = (\cos \phi_{i}, \sin \phi_{i}), \quad \phi_{i} = 2\pi ( i-1)/N, \quad i = 1, \dots, N
\]  
such that $N=32$ in all of our numerical examples. Here $\eta$ is a complex--valued random vector such that the real and imaginary parts of its entries are uniformly distributed on the interval $[-1,1]$ and $0<\delta<1$ denotes the relative noise level. Here, we let $\| \cdot \|_{2}$ denote the standard 2--norm in $\mathbb{C}^{N}$.

Now that we know how the scattering data is computed we need to discuss how the spectral data used to define the imaging functional \eqref{reg_W} and \eqref{reg_multi_W} is computed. To this end, we need to discretize the Dirichlet or Neumann far--field operator such that 
finite--dimensional far--field matrix $\mathbf{F}^\sigma_z \in \mathbb{C}^{N \times N}$ is defined as  
\[
\mathbf{F}^\sigma_z (i,j) = \mathrm e^{-\mathrm{i} \kappa  z \cdot (\hat x_{i} - \hat y_{j}) } U_{\sigma,B_0}^\infty(\hat x_{ i}, \hat y_{ j}), \quad i,j = 1, \dots, N,
\]  
for $\sigma = \{$Dir,Neu$\}$ with equally spaced observation and reference directions  
\[
\hat x_{i} = \hat y_{i} = (\cos \phi_{i}, \sin \phi_{i})^\top, \quad \phi_{i} = 2\pi ( i-1)/N, \quad i = 1, \dots, N.
\]  
Again in all of our examples $N=32$ for simplicity. For the Dirichlet and Neumann far--field operators the kernel function will be approximated via 
$$U_{\mathrm{Dir} , B_0}^{\infty}(\hat x,\hat y)= -\frac{4}{\text{i}}\left[ \frac{J_0(\kappa R)}{H_0^{(1)}(\kappa R)} + 2\sum_{n=1}^{10}\frac{J_n(\kappa R)}{H_n^{(1)}(\kappa R)} \cos\big(n(\theta_x-\theta_y)\big) \right]$$
and 
$$U_{\mathrm{Neu} , B_0}^{\infty}(\hat x,\hat y)= -\frac{4}{\text{i}}\left[ \frac{J_1(\kappa R)}{H_1^{(1)}(\kappa R)} + 2\sum_{n=1}^{10}\frac{J_{n-1}(\kappa R) - J_{n+1}(\kappa R)}{H_{n-1}^{(1)}(\kappa R)-H_{n+1}^{(1)}(\kappa R)} \cos\big(n(\theta_x-\theta_y)\big) \right].$$
This corresponds to truncating the series and only using the first 11--Fourier modes.

In each example, we employ an equally spaced $100 \times 100$ sampling grid over the imaging domain $[-4,4]^2$. So for each sampling point $z$ and incident direction $d_{\ell}$, the corresponding regularized imaging functional derived from \eqref{reg_W} is computed via 
\begin{align*}\label{single_W_finite_noise}
W(z; d_{\ell}) =
\left[ \sum_{j=1}^{N}
 \mathds{1}_{[\alpha, \infty)} \big( |\lambda_j| \big) \frac{\big| u_\delta^\infty(\cdot \, , d_{\ell}) \cdot \mathbf{v}_z^{(j)} \big|^2}{|\lambda_z^{(j)}|} \right]^{-1},
\end{align*}  
where $(\lambda_z^{(j)}, \mathbf{v}_z^{(j)})$ are the eigenvalues and eigenvectors of the far--field matrix $\mathbf{F}^\sigma_z$, obtained via
\[
\mathbf{F}^\sigma_z = \mathbf{V}_z \mathbf{\Lambda}_z \mathbf{V}_z^*, \quad \mathbf{\Lambda}_z = \mathrm{diag}\left(\lambda_z^{(1)}, \dots, \lambda_z^{(N)}\right),
\]
and the sum is taken over the retained modes after truncation with threshold $\alpha>0$. To incorporate multiple incident directions $\{d_\ell\}_{\ell=1}^{N_{\mathrm{inc}}}$, from \eqref{reg_multi_W}, the contributions are aggregated via  
\begin{equation}\label{multi_W_finite_noise}
W_{\mathrm{Recon}}(z) = \sum_{\ell=1}^{N_{\mathrm{inc}}} W(z; d_\ell).
\end{equation}  
Here is the algorithm for the reconstructions.

\begin{algorithm}[H]\label{alg_esm_refined}
\caption{Factorization--Based ESM with Radius Refinement}
\begin{algorithmic}[1]

\STATE \textbf{Input:} initial radius $R_0 > 0$, growth factor $\gamma > 1$, sampling grid $\mathcal{M}$, parameter $\alpha>0$, far--field data $u_\delta^\infty$.

\STATE \textbf{Output:} Contour plot of the imaging functional that approximates the clamped cavity $D$.

\FOR{$p = 0,1,\dots,P_{\text{max}}$}
    \STATE Set radius: $R_p = \gamma^p R_0.$
    \FOR{each sampling point $z \in \mathcal{M}$}
        \STATE Assemble the far--field matrix $\mathbf{F}^\sigma_z$ and compute $[\mathbf{V}_z,\mathbf{\Lambda}_z]$ = eigs$(\mathbf{F}^\sigma_z)$.       
        \STATE Evaluate $W_{\mathrm{Recon}}(z)$ from \eqref{multi_W_finite_noise} at each sampling point $z \in \mathcal{M}$.
    \ENDFOR
    \STATE Check contour plot of the imaging functional for artifacts.
    \STATE {\bf If} the reconstruction has any artifacts set $p=p+1$ and return to Step 4
    \STATE {\bf Else} terminate the iteration and proceed to Step 13
\ENDFOR
\STATE \textbf{Return:}  Contour plot of the imaging functional.
\end{algorithmic}
\end{algorithm}

For our numerical examples, we consider three sets of incident directions that will produce the clamped scattering data. Here, the three subsets of the unit circle we consider are 
\begin{align*}
\mathbb S_{\text{inc},1}^1&=\big\{ d=(\cos \phi, \sin \phi ) \,\, : \,\, \phi= \pi \big\},\\
\mathbb S_{\text{inc},2}^1&=\big\{ d=(\cos \phi, \sin \phi ) \,\, : \,\, \phi= j\pi/4,  j=1,3 \big\},\\
\mathbb S_{\text{inc},4}^1&=\big\{ d=(\cos \phi, \sin \phi) \,\, : \,\, \phi= j\pi/2,  j=0,1,2,3 \big\}
\end{align*}
i.e. we consider the case when $N_{\text{inc}}$=1, 2 or 4. We will use $N=32$ observation directions in all reconstructions. For all ESM reconstructions, the regularization parameter is empirically chosen to be $\alpha=10^{-4}$. With this, we have all we need to proceed with our numerical experiments.

\subsection{Effect of the radius for the sampling disk}
We first examine the role of the radius for the sampling disk $R$ in the performance of the factorization--based ESM. The choice of $R$ is critical, as it governs the scale of the auxiliary test domain used in the reconstruction procedure. At first glance, one might assume that the radius $R$ should be a small quantity but in our experiments we see that this is not the case. This seems to be due to the fact that each fixed eigenvalue of the Dirichlet and Neumann far--field operators tend to zero as $R \to 0$.

As previously stated, from our testing we see that the sampling disk with $R \approx 0.5\cdot\mathrm{diam}(D)$ gives optimal results. In order to validate our claim, we provide a few numerical examples in this section. In practice, we adopt an adaptive strategy consistent with the given algorithm. We select radii from the sequence $R_p = (1.25)^p \cdot 0.5$, starting from an initial value and increasing $p$ until a stable reconstruction is observed. This procedure is motivated by the heuristic principle that an appropriate reconstruction is obtained when the diameter of the sampling disk is comparable to that of the unknown scatterer.

We begin with reconstructions of the peanut--shaped obstacle. In Figure \ref{findR-peanutSS} and \ref{findR-peanutSH}, we consider the case of two incident directions corresponding to $\mathbb S_{\text{inc},2}^1$ where the obstacle is centered at the point $(1,1)$. This corresponds to the boundary of the scatterer being given by 
$$\partial D  = \left(\tfrac{1}{2}\sqrt{3\cos^2 t + 1}\right)(\cos t,\sin t)+(1,1)$$ 
Here we fix the wavenumber $\kappa = 2$. First we investigate the effect of the radius using the sound--soft (Dirichlet) sampling disk in Figure \ref{findR-peanutSS} then provide the reconstructions using the sound--hard (Neumann) sampling disk in Figure \ref{findR-peanutSH}. In both sets of reconstructions, we take the noise level in the data to be $\delta = 0.02$.

\begin{figure}[htp]
    \centering
    \begin{subfigure}[t]{0.3\textwidth}
        \centering
       \includegraphics[scale=0.2]{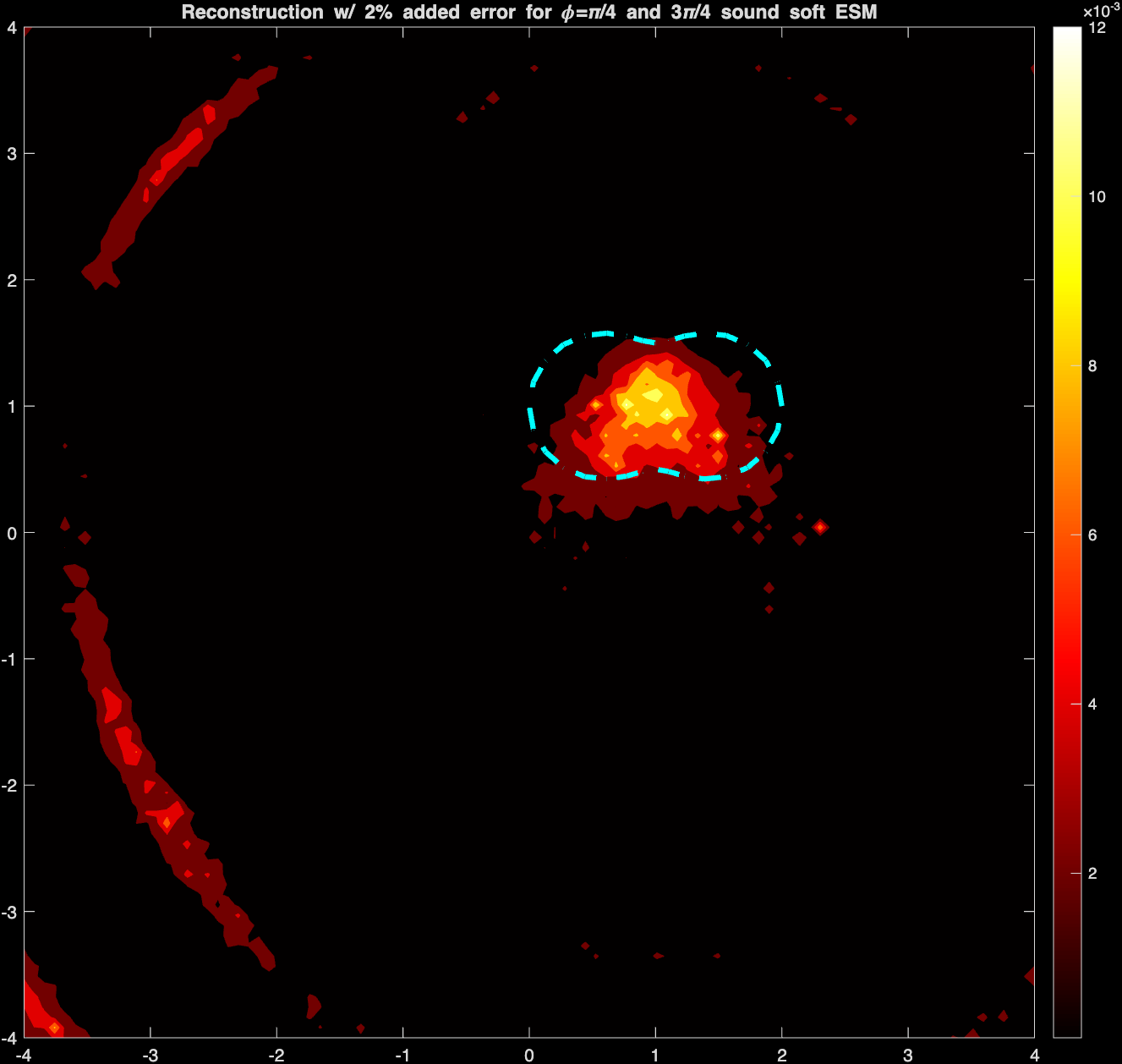}
        \caption{$p=0$ ;  $R=0.5$}
    \end{subfigure}%
    ~ 
    \begin{subfigure}[t]{0.3\textwidth}
        \centering
        \includegraphics[scale=0.2]{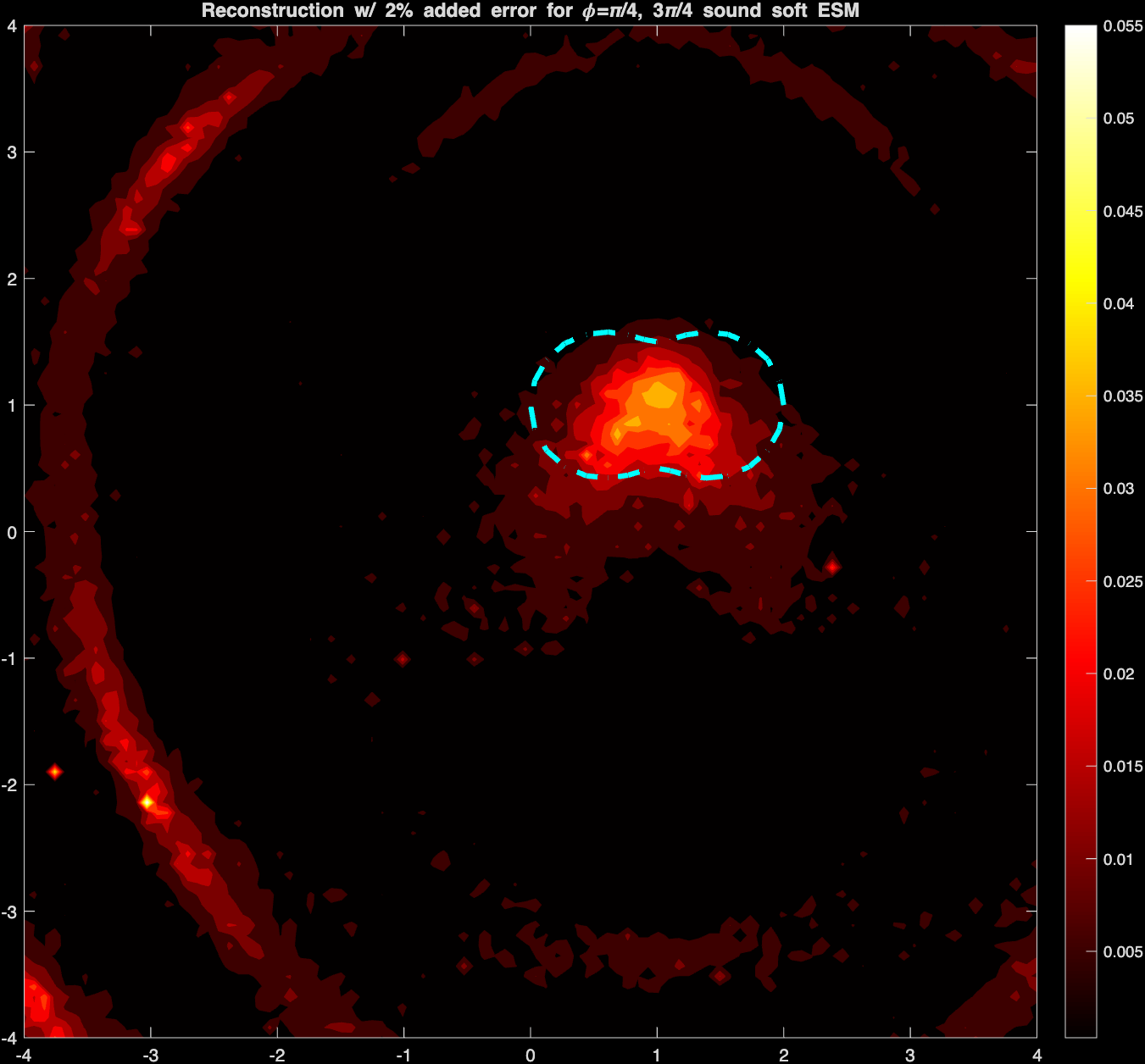}
        \caption{$p=1$ ;  $R=0.6250$}
    \end{subfigure}%
       ~ 
    \begin{subfigure}[t]{0.3\textwidth}
        \centering
        \includegraphics[scale=0.2]{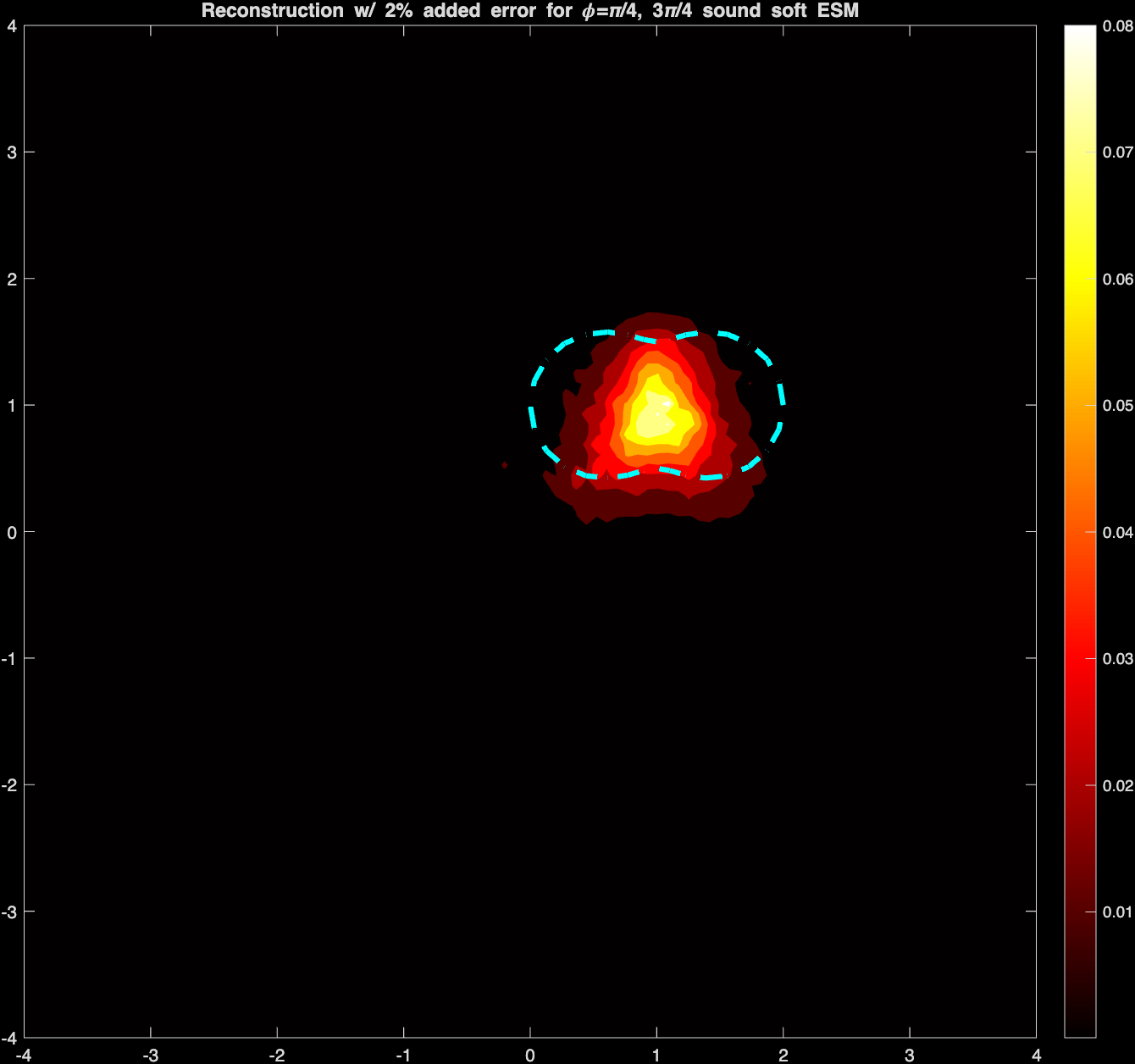}
        \caption{$p=2$ ;  $R=0.7812$}
    \end{subfigure}%
    \caption{Reconstruction of a peanut--shaped obstacle centered at $(1,1)$ with 2$\%$ added noise using the sound--soft (Dirichlet) factorization--based ESM, with two incident directions. }\label{findR-peanutSS}
\end{figure}

\begin{figure}[htp]
    \centering
    \begin{subfigure}[t]{0.3\textwidth}
        \centering
       \includegraphics[scale=0.2]{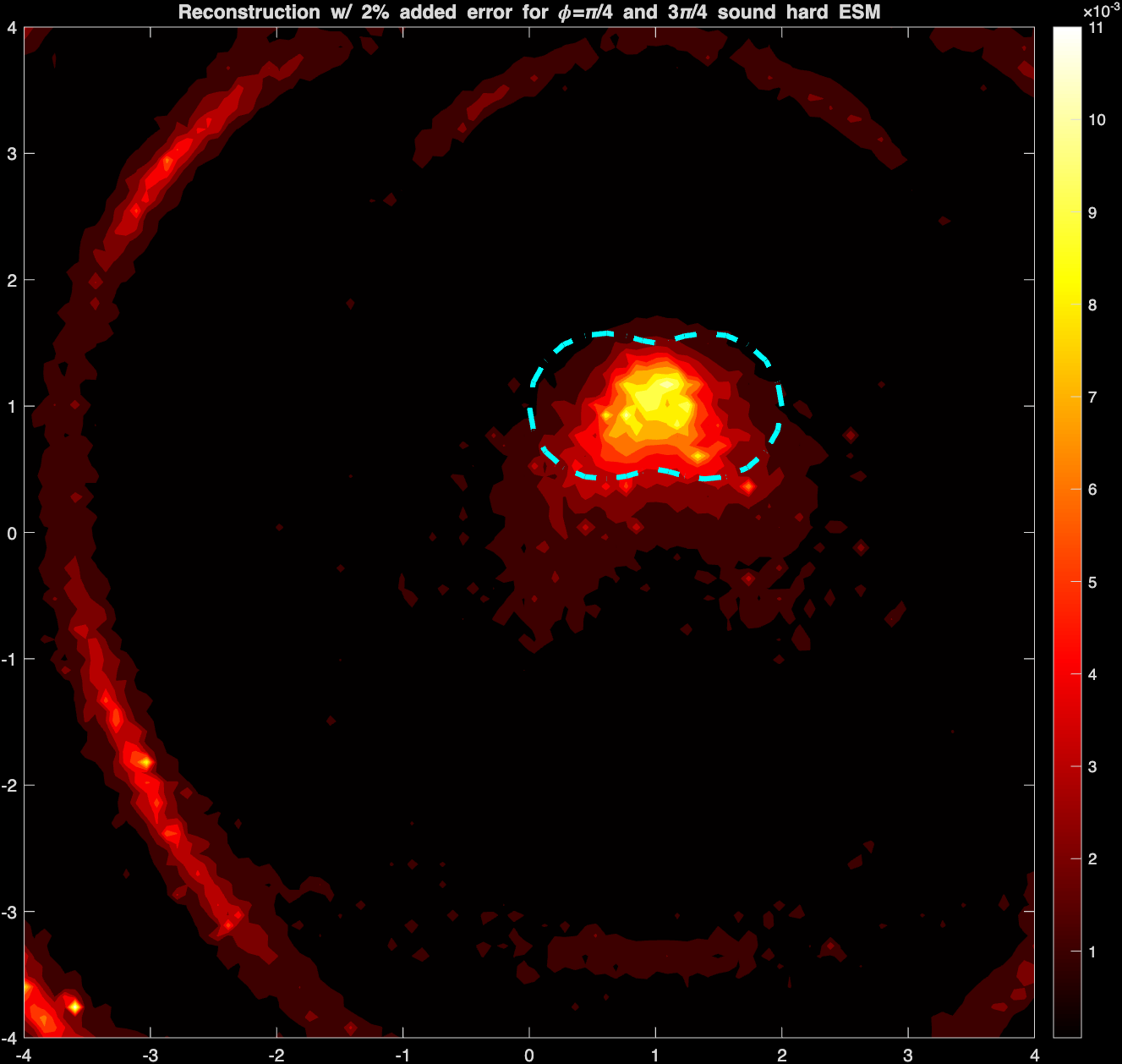}
        \caption{$p=0$ ;  $R=0.5$}
    \end{subfigure}%
    ~ 
    \begin{subfigure}[t]{0.3\textwidth}
        \centering
        \includegraphics[scale=0.2]{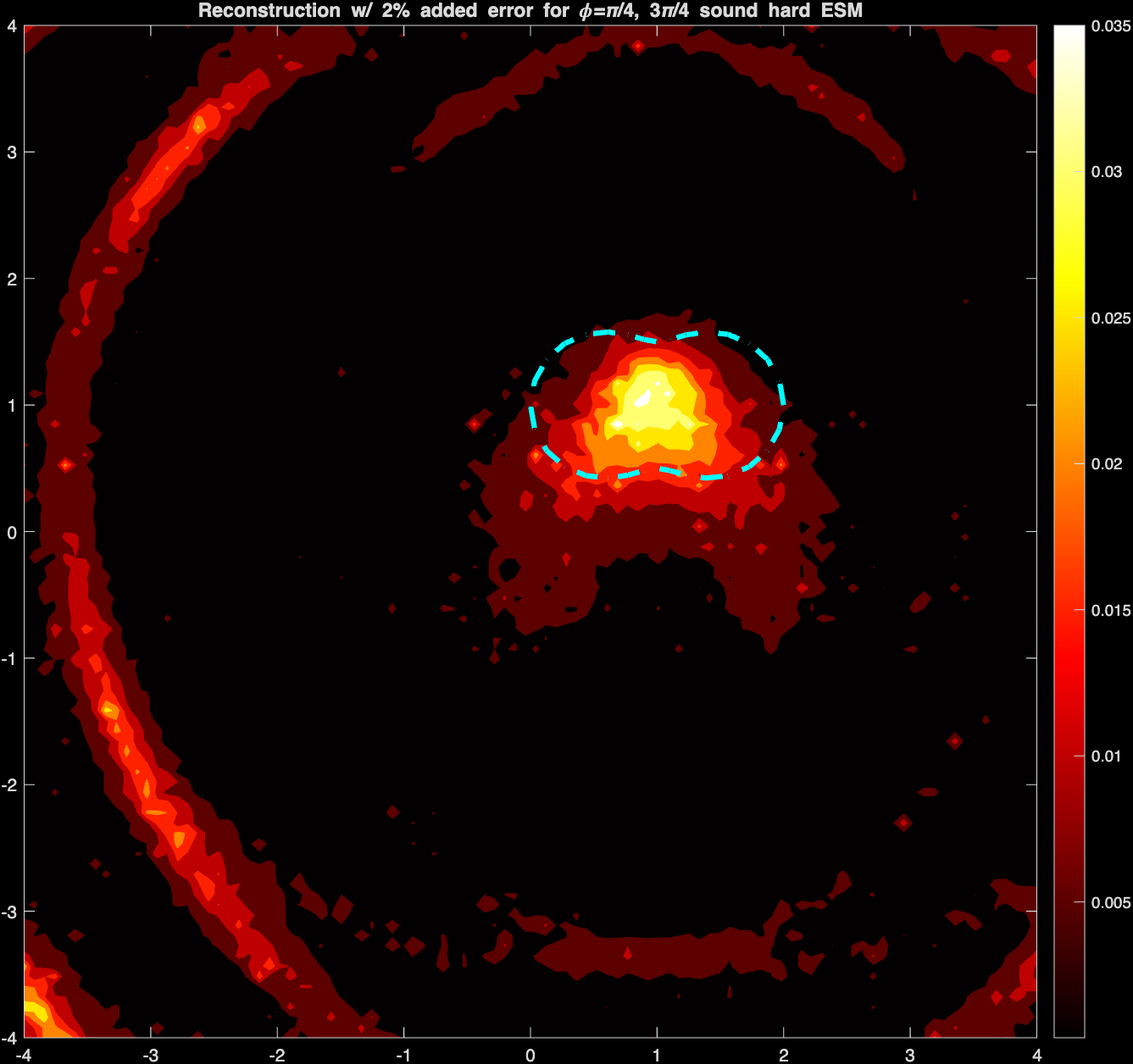}
        \caption{$p=1$ ;  $R=0.6250$}
    \end{subfigure}%
       ~ 
    \begin{subfigure}[t]{0.3\textwidth}
        \centering
        \includegraphics[scale=0.2]{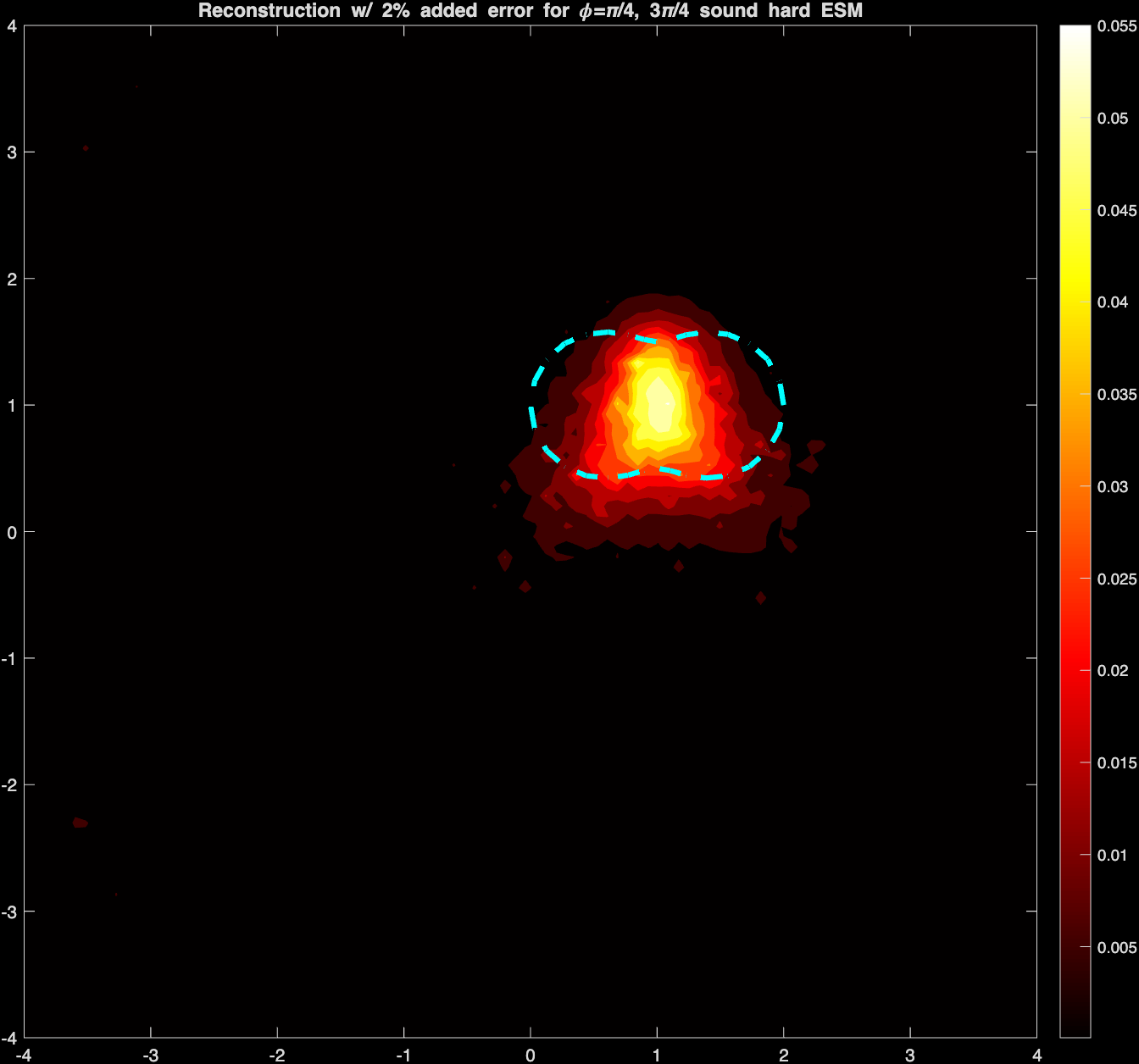}
        \caption{$p=2$ ;  $R=0.7812$}
    \end{subfigure}%
    \caption{Reconstruction of a peanut--shaped obstacle centered at $(1,1)$ with 2$\%$ added noise using the sound--hard (Neumann) factorization--based ESM, with two incident directions. }\label{findR-peanutSH}
\end{figure}

From Figure \ref{findR-peanutSS} and \ref{findR-peanutSH}, we see that when the diameter of the sampling disk is sufficiently smaller than the diameter of $D$, that the reconstruction has `artifacts'. When the diameter of the sampling disk is comparable to the diameter of $D$, that the reconstruction is localized inside the scatterer. This is useful since we can also get a sense of the size of the scatterer using this procedure.

We now provide reconstructions of the star--shaped obstacle centered at  $(1,1)$. In Figure \ref{findR-starSH} we consider the case of a single incident direction corresponding to $\mathbb S_{\text{inc},1}^1$, where the reconstructions are obtained using the sound--hard (Neumann) sampling disk. This corresponds to the boundary of the scatterer being given by 
$$\partial D  = \big(1 + 0.3\cos(4t)\big)(\cos t,\sin t)+(1,1)$$ 
Here we fix the wavenumber $\kappa = 2$ and we take the noise level in the data to be $\delta = 0.02$ as in the previous examples. Again, we see that when the sampling diameter is comparable to the diameter of $D$, the reconstructions are better and contain no artifacts. Due to the use of a single incident direction, the reconstruction exhibits a directional bias. The portion of the obstacle facing the incoming wave is more clearly resolved, while regions in the shadow of the wave are less visible.

\begin{figure}[H]
    \centering
    \begin{subfigure}[t]{0.3\textwidth}
        \centering
       \includegraphics[scale=0.2]{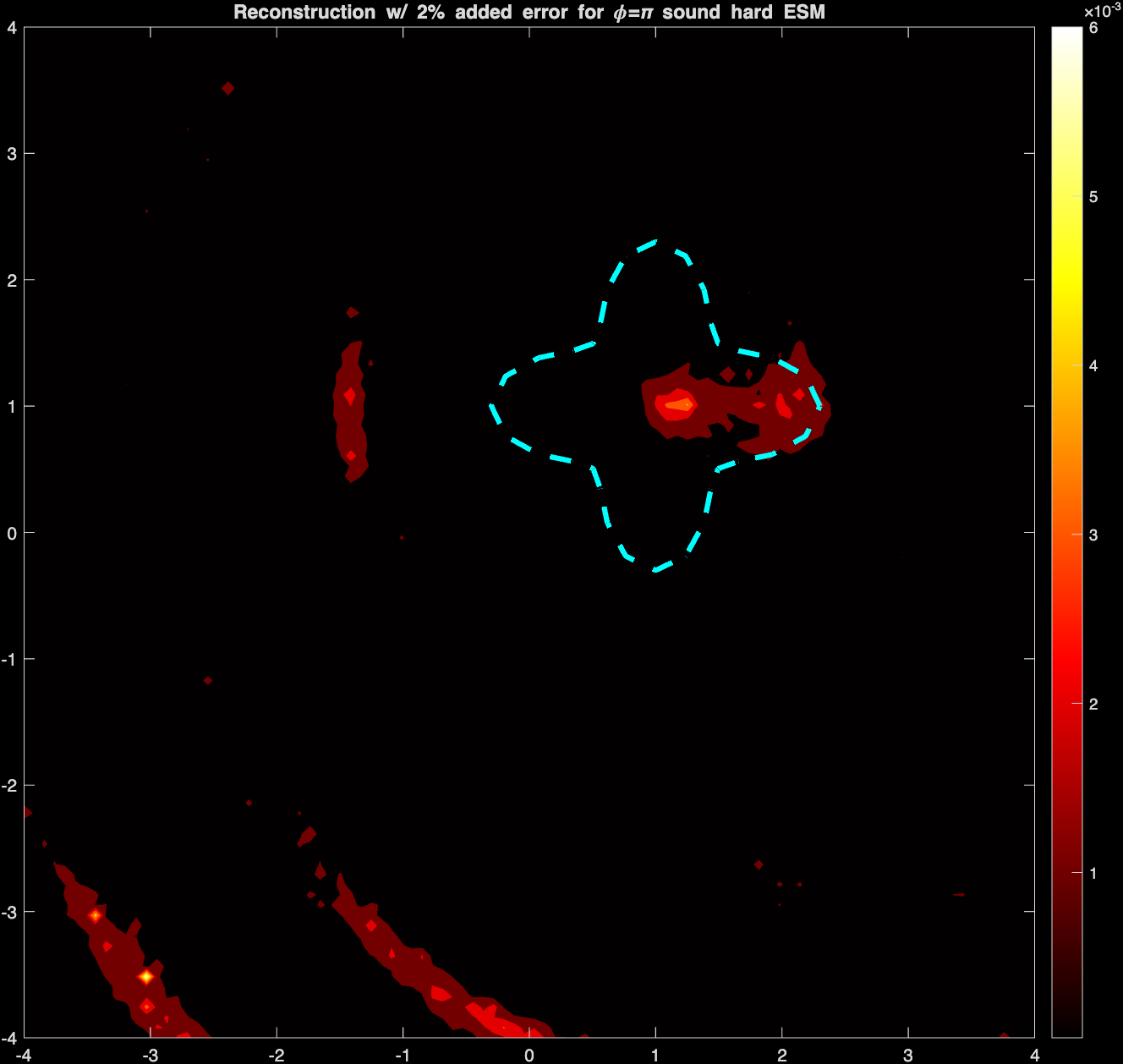}
        \caption{$p=0$ ;  $R=0.5$}
    \end{subfigure}%
    ~ 
    \begin{subfigure}[t]{0.3\textwidth}
        \centering
        \includegraphics[scale=0.2]{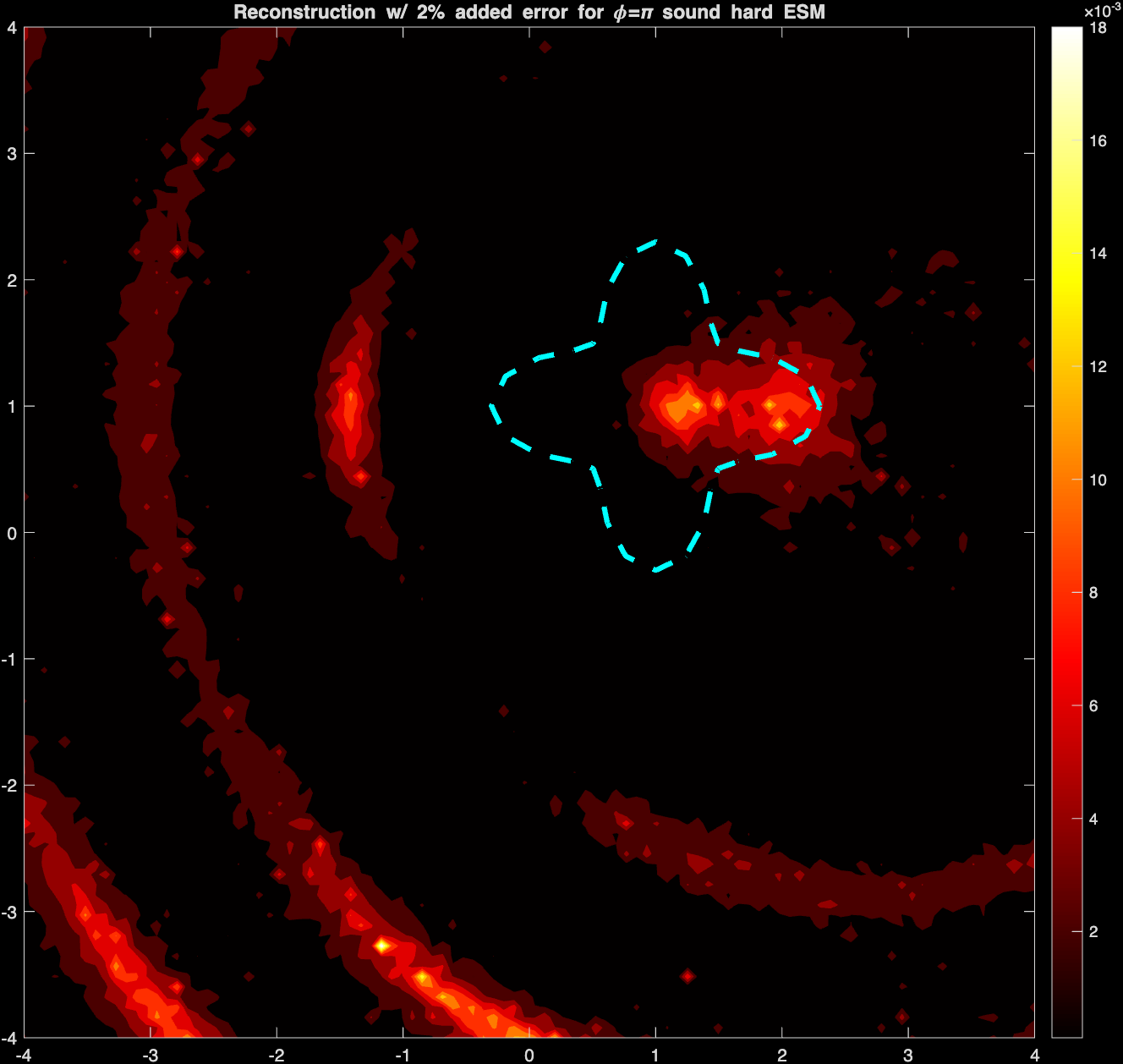}
        \caption{$p=1$ ;  $R=0.6250$}
    \end{subfigure}
       ~ 
    \begin{subfigure}[t]{0.3\textwidth}
        \centering
        \includegraphics[scale=0.2]{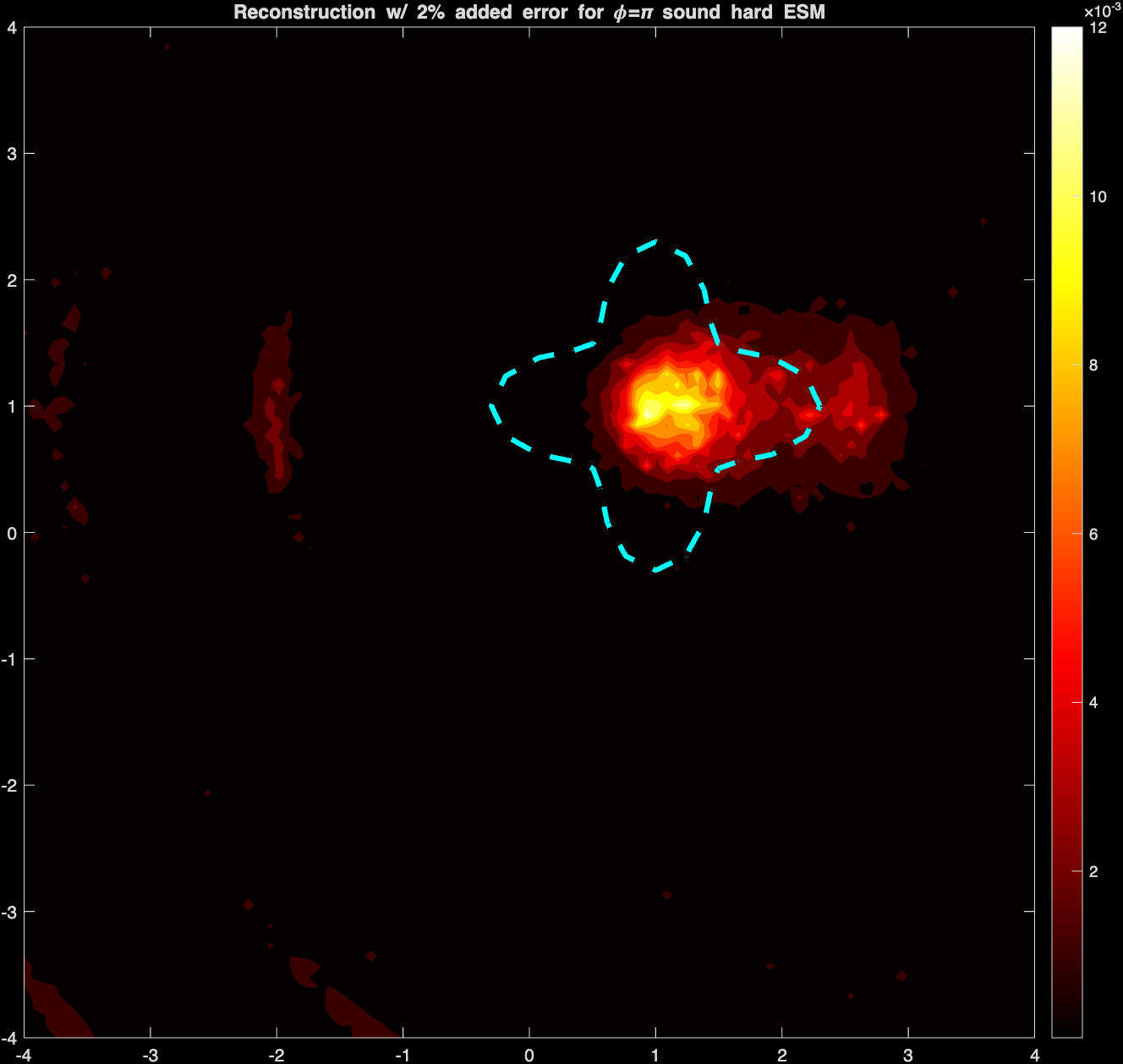}
        \caption{$p=2$ ;  $R=0.7812$}
    \end{subfigure}%
       ~ 
    \begin{subfigure}[t]{0.3\textwidth}
        \centering
        \includegraphics[scale=0.2]{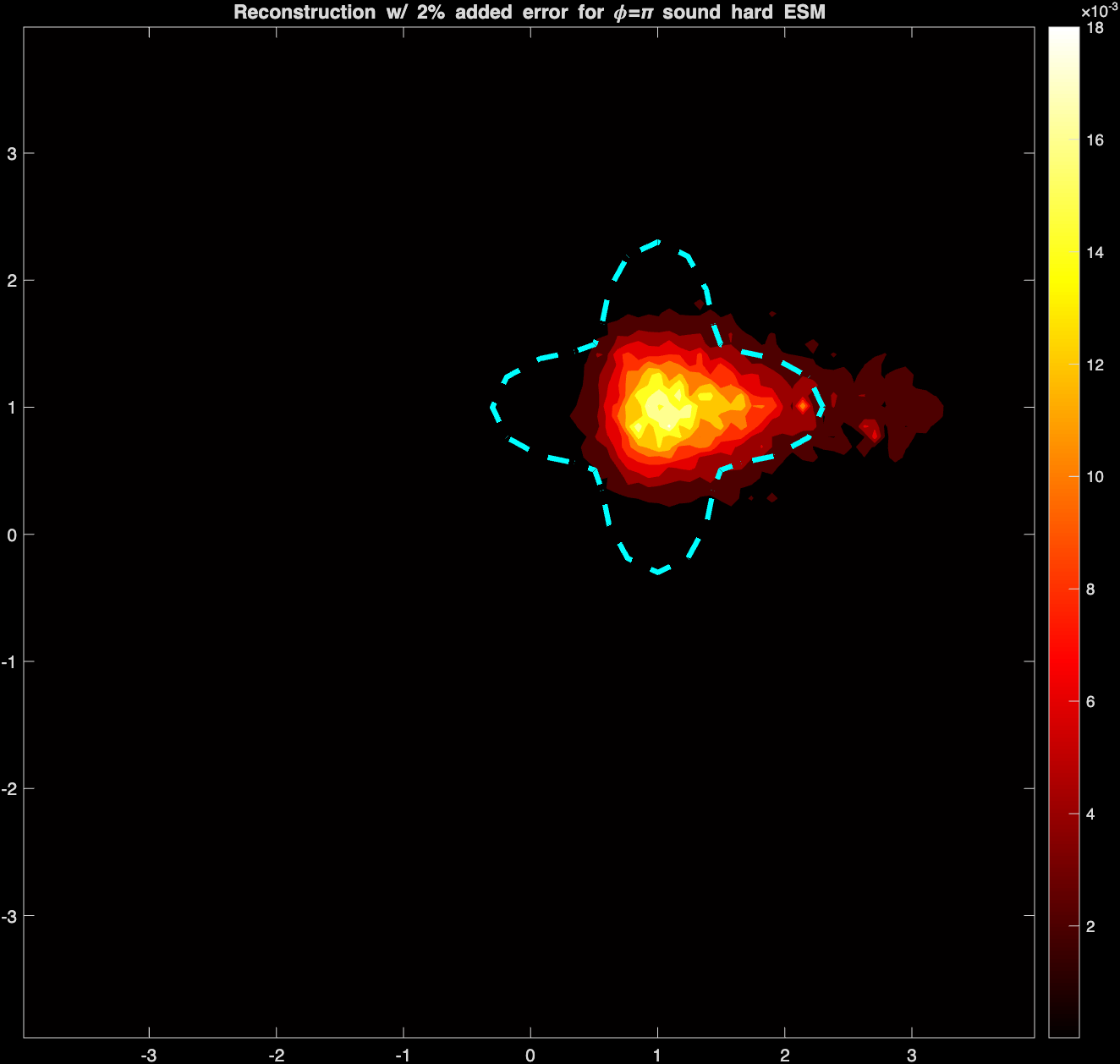}
        \caption{$p=3$ ;  $R=0.9766$}
    \end{subfigure}
    \caption{Reconstruction of a star--shaped obstacle centered at (1,1) with 2$\%$ added noise using the sound--hard (Neumann) factorization--based ESM, with one incident direction. }\label{findR-starSH}
\end{figure}

We now consider the case of four incident directions corresponding to $\mathbb S_{\text{inc},4}^1$. This corresponds to the dataset that has the most information. This is because we assume that the scattering data is known at $N=32$ receivers placed all around the scatterer and produced from $N_{\text{inc}}=4$ sources. Notice that the four sources in $\mathbb S_{\text{inc},4}^1$ correspond to the axes in $\R^2$. This can be seen as information collected from all around the scatterer but from a small number of sources. In Figure \ref{findR-starSS}, we will again consider the star--shaped obstacle which is now centered at $(-1,-1)$   given by 
$$\partial D  = \big(1 + 0.3\cos(4t)\big)(\cos t,\sin t)-(1,1).$$ 
Here we take sound--soft (Dirichlet) sampling disk for our reconstruction with the wavenumber $\kappa = \pi$. We take a larger noise level in the data, which is given by $\delta = 0.1$ to test the stability of the reconstructions. 

\begin{figure}[htp]
    \centering
    \begin{subfigure}[t]{0.3\textwidth}
        \centering
       \includegraphics[scale=0.244]{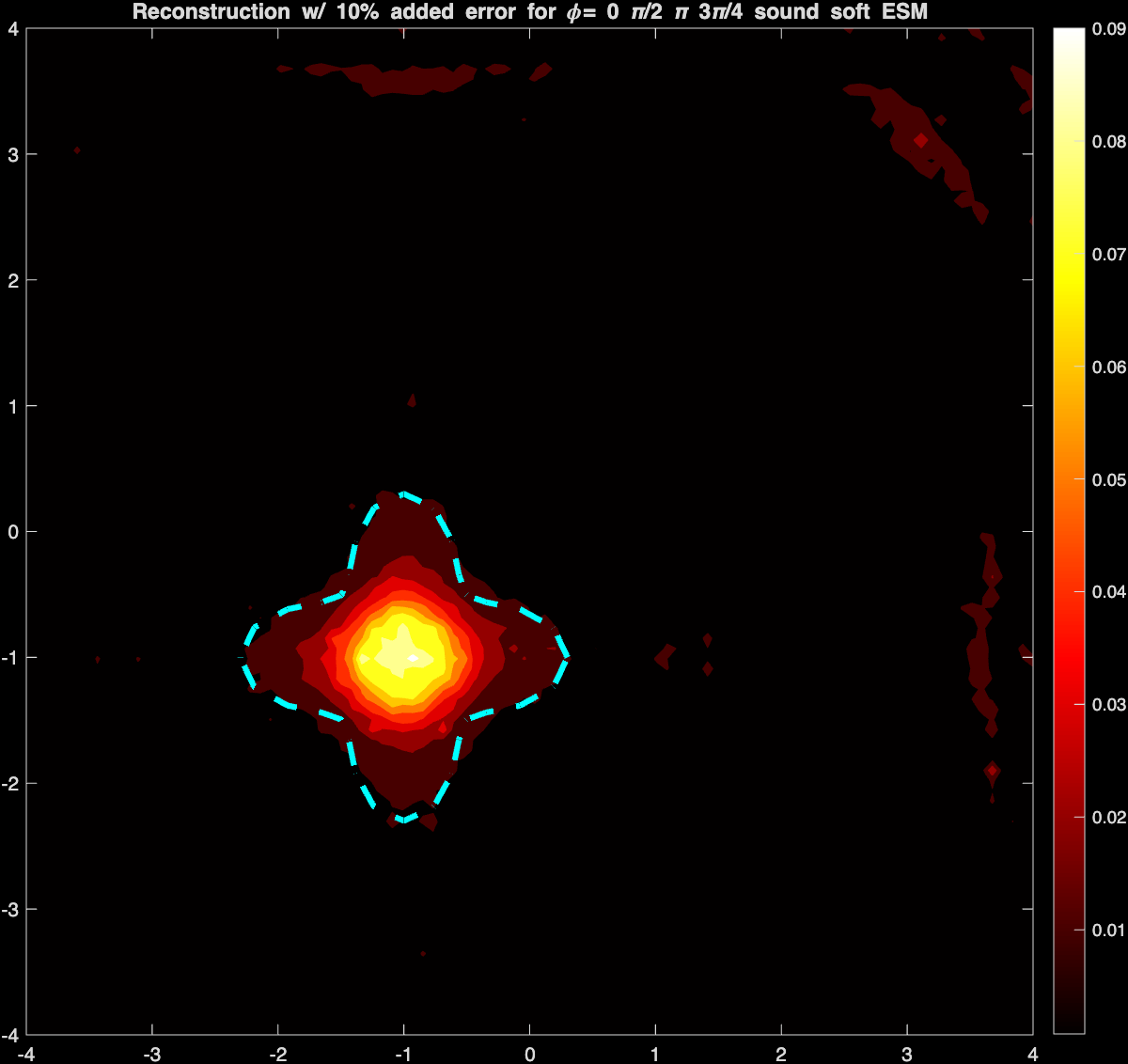}
        \caption{$p=3$ ;  $R=0.9766$}
    \end{subfigure}%
    ~ 
    \begin{subfigure}[t]{0.3\textwidth}
        \centering
        \includegraphics[scale=0.244]{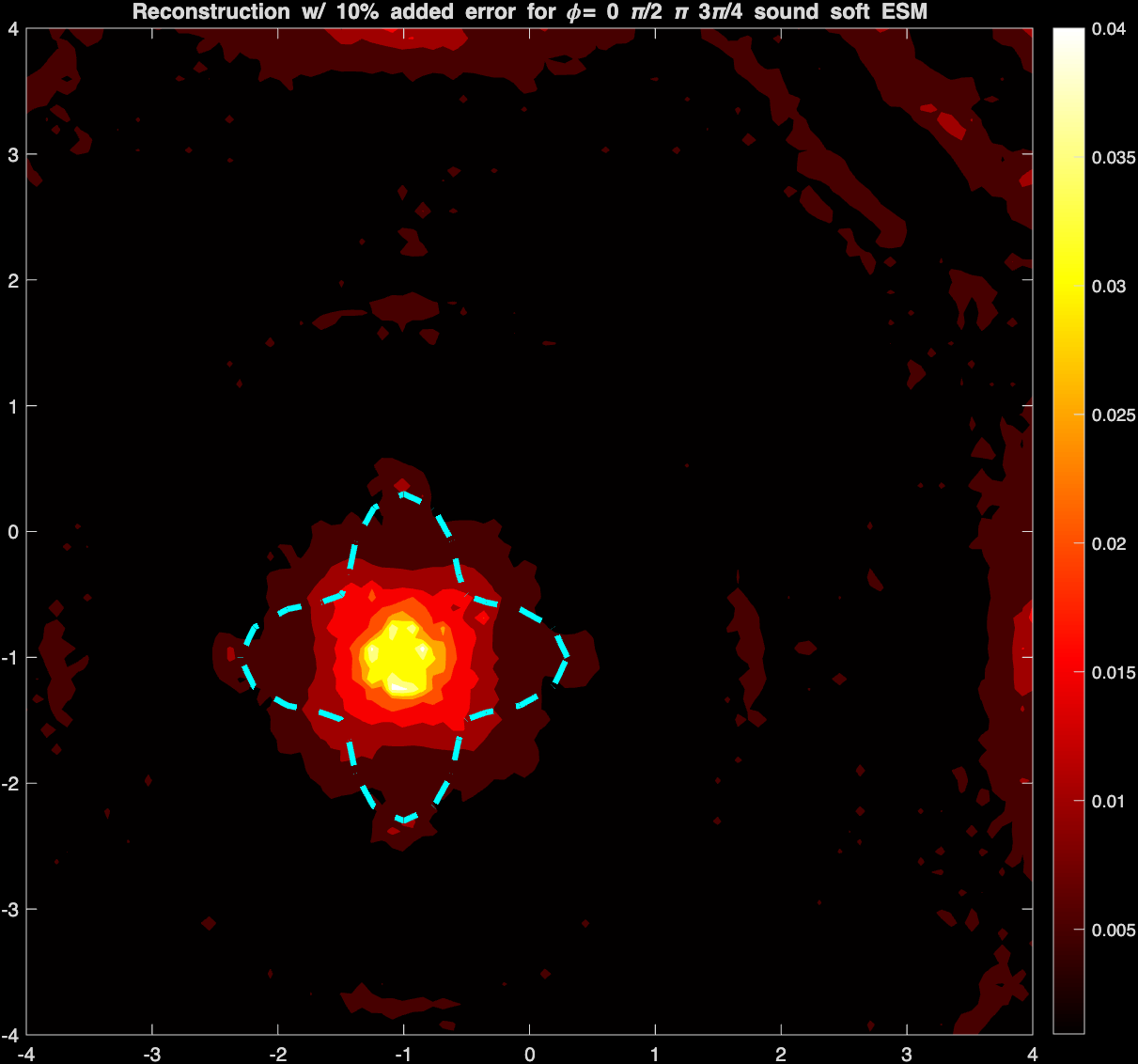}
        \caption{$p=4$ ;  $R=1.2207$}
    \end{subfigure}%
       ~ 
    \begin{subfigure}[t]{0.3\textwidth}
        \centering
        \includegraphics[scale=0.22]{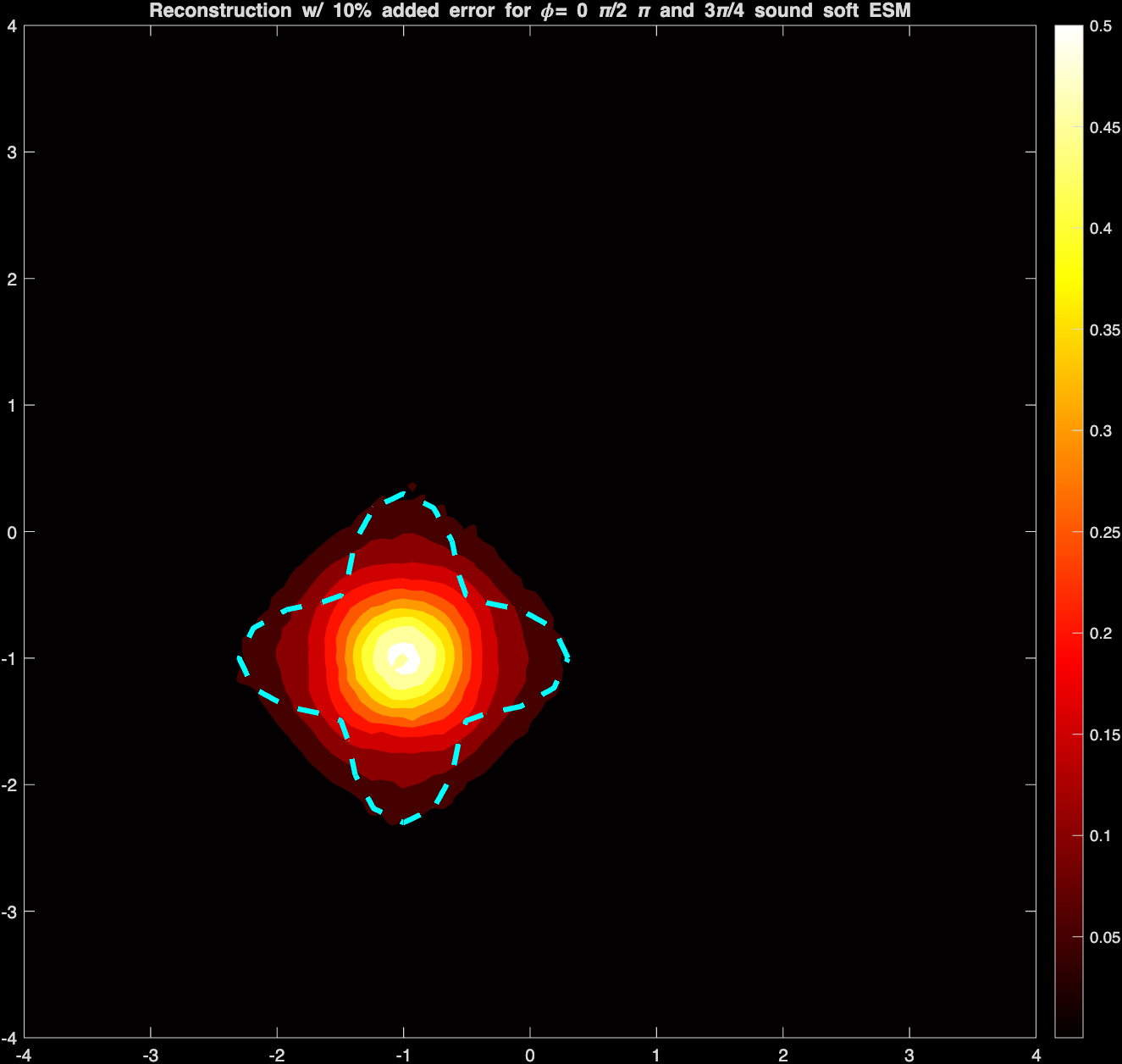}
        \caption{$p=5$ ;  $R=1.5259$}
    \end{subfigure}%
    \caption{Reconstruction of a star--shaped obstacle centered at $(-1,-1)$ with 10$\%$ added noise using the sound--soft (Dirichlet) factorization--based ESM, with four incident directions.}\label{findR-starSS}
\end{figure}

These results illustrate the influence of both the number of incident directions and the choice of reference disk on reconstruction stability under noise. In all cases, the obstacle location is accurately identified and remains robust with respect to noise and reference type. Notably, incorporating two or four incident directions leads to improved stability, with the indicator more tightly concentrated around the true obstacle location for both sound--soft and sound--hard configurations. Although the use of multiple incident directions improves the reconstruction and reduces directional bias, it does not remove the dependence on the probing radius. Rather, these two factors play complementary roles: the radius governs the effective scale of the reconstruction, while the incident directions enhance angular resolution and overall stability.

\subsection{Comparison using sound--soft vs sound--hard sampling disk}

Now we wish to investigate how our ESM compares when one uses either the sound--soft or sound--hard sampling disk. As we have seen in our previous examples, both cases can successfully identify the location of the scatterer. In particular, we see that when one has data from multiple incident directions the size and shape of the scatterer can also be approximated. Here we wish to see if there is a noticeable difference in using either sound--soft or sound--hard sampling disk. To this end, we use the same procedure described in the previous section to reconstruct the scatterers. In the following examples, we consider that the peanut--shaped scatterer is centered at $(1,1)$ and the star--shaped scatterer is centered at $(-1,-1)$. 

We first start with an example where we use the set of incident directions given by $\mathbb{S}_{\text{inc},4}^1$. As we saw in the previous examples, even with only four incident directions we have a fairly good reconstruction of the location and size of the scatterer. In that example, the sound--soft sampling disk was used and the optimal radius was given by $R=1.5259$. Now, in Figure \ref{compare-star4inc} we compare the reconstructions sound--soft or sound--hard sampling disk for the star--shaped scatterer. We again use the aforementioned procedure to determine the optimal sampling radius. For this numerical example we let the wavenumber $\kappa=\pi$ and $\delta = 0.1$ which corresponds to $10\%$ relative noise added to the data. Here, we only present the reconstructions given by the optimal radius. 

\begin{figure}[H]
    \centering
    \begin{subfigure}[t]{0.4\textwidth}
        \centering
        \includegraphics[scale=0.24]{Star_SS4Inc_p5_kpi}
        \caption{sound--soft reconstruction w/ $\mathbb S_{\text{inc},4}^1$ }
    \end{subfigure}%
       ~ 
    \begin{subfigure}[t]{0.4\textwidth}
        \centering
        \includegraphics[scale=0.24]{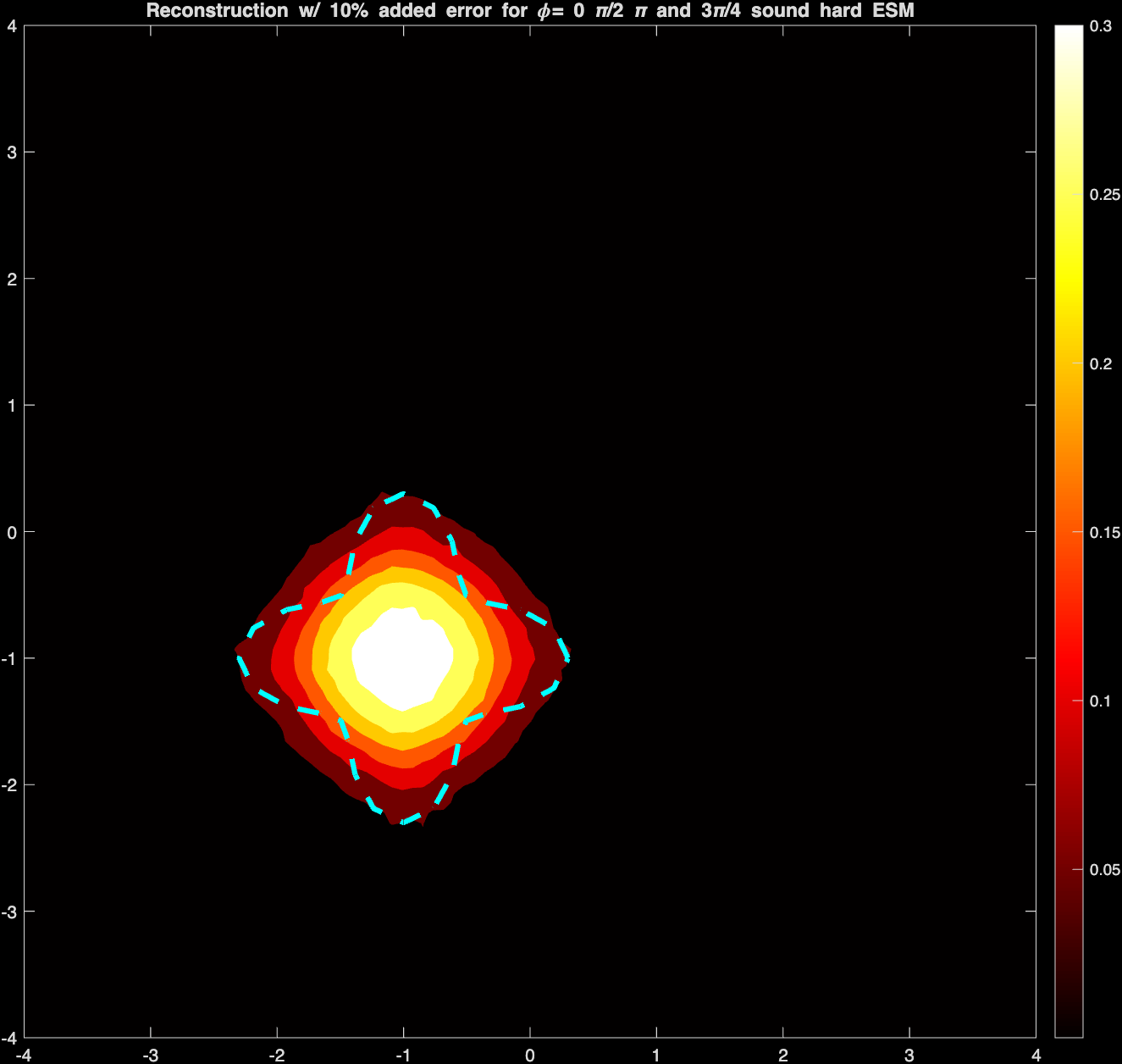}
        \caption{sound--hard reconstruction w/ $\mathbb S_{\text{inc},4}^1$}
    \end{subfigure}%
    \caption{Comparison of the reconstruction given by the sound--soft and sound--hard sampling disks for the star--shaped obstacle centered at $(-1,-1)$ with 10$\%$ added noise factorization--based ESM, with four incident directions.}\label{compare-star4inc}
\end{figure}

We again see that with four incident directions, we can obtain a good reconstruction of the location and size of the obstacle. This is useful since most qualitative reconstruction methods require data from a large number of incident directions. There have been recent theoretical and computational works on the case when one does not have the full multi--static dataset (for, e.g., \cite{DouLiuMengZhang2022,LiuSun2019}). Here we see that this gives the best reconstruction since one has data from the incident fields all around the object, but only a small number of directions.

In Figure \ref{compare-peanut4inc}, we provide the reconstruction of the peanut--shaped obstacle centered at $(1,1)$. We compare the reconstructions with sound--soft and sound--hard sampling disks. In this example, the wavenumber $\kappa=2$ and $\delta = 0.02$ which corresponds to only having $2\%$ relative noise added to the data. As in our previous comparison, we only present the reconstructions given by the optimal radius which is, in both cases, $R=0.7812$. Note that we see that the reconstructions are comparable for both cases when one uses the set of incident directions given by $\mathbb{S}_{\text{inc},4}^1$. Here, we only present the reconstructions given by the optimal radius. 

\begin{figure}[H]
    \centering
    \begin{subfigure}[t]{0.4\textwidth}
        \centering
        \includegraphics[scale=0.24]{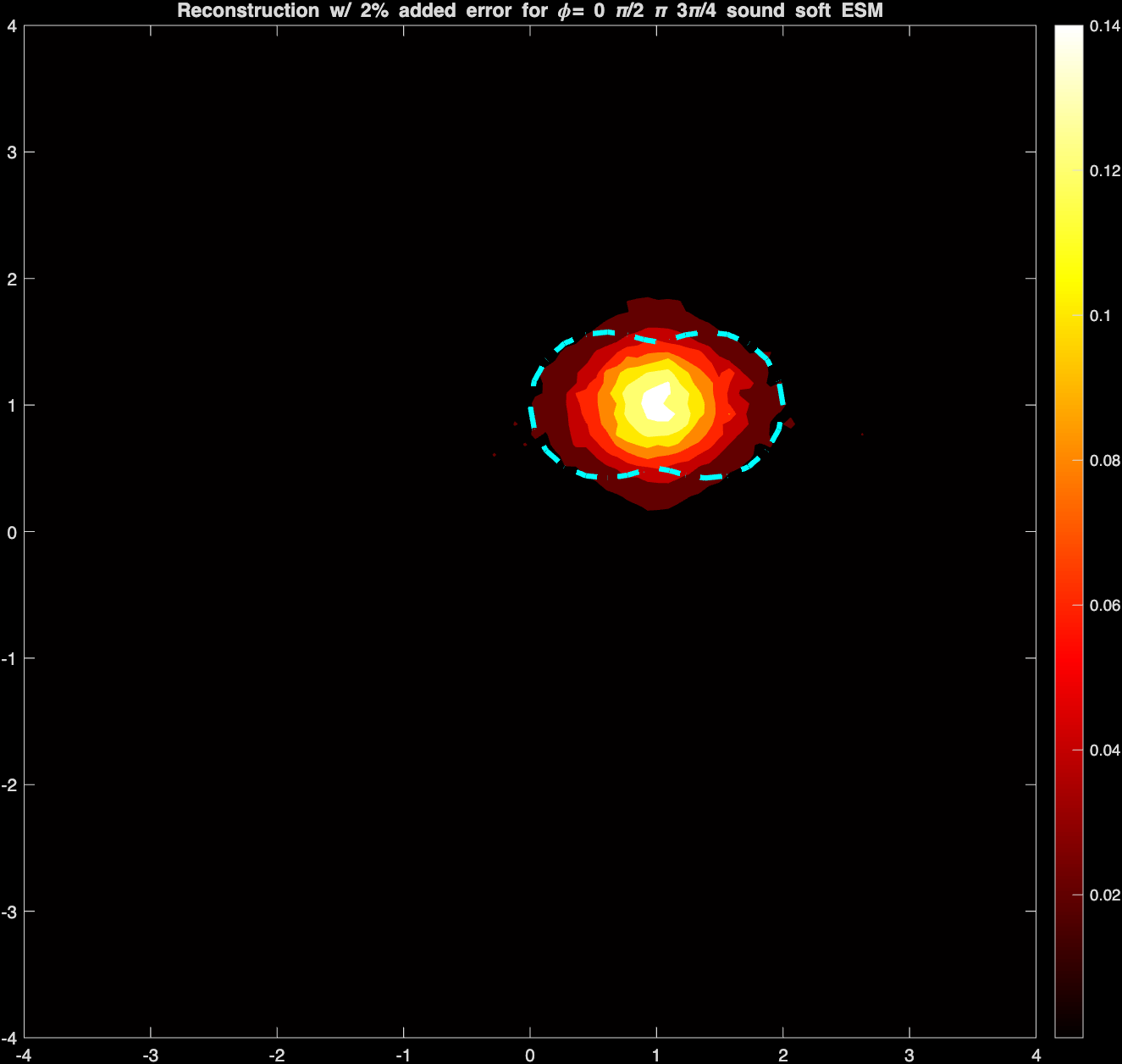}
        \caption{sound--soft reconstruction w/ $\mathbb S_{\text{inc},4}^1$}
    \end{subfigure}%
       ~ 
    \begin{subfigure}[t]{0.4\textwidth}
        \centering
        \includegraphics[scale=0.24]{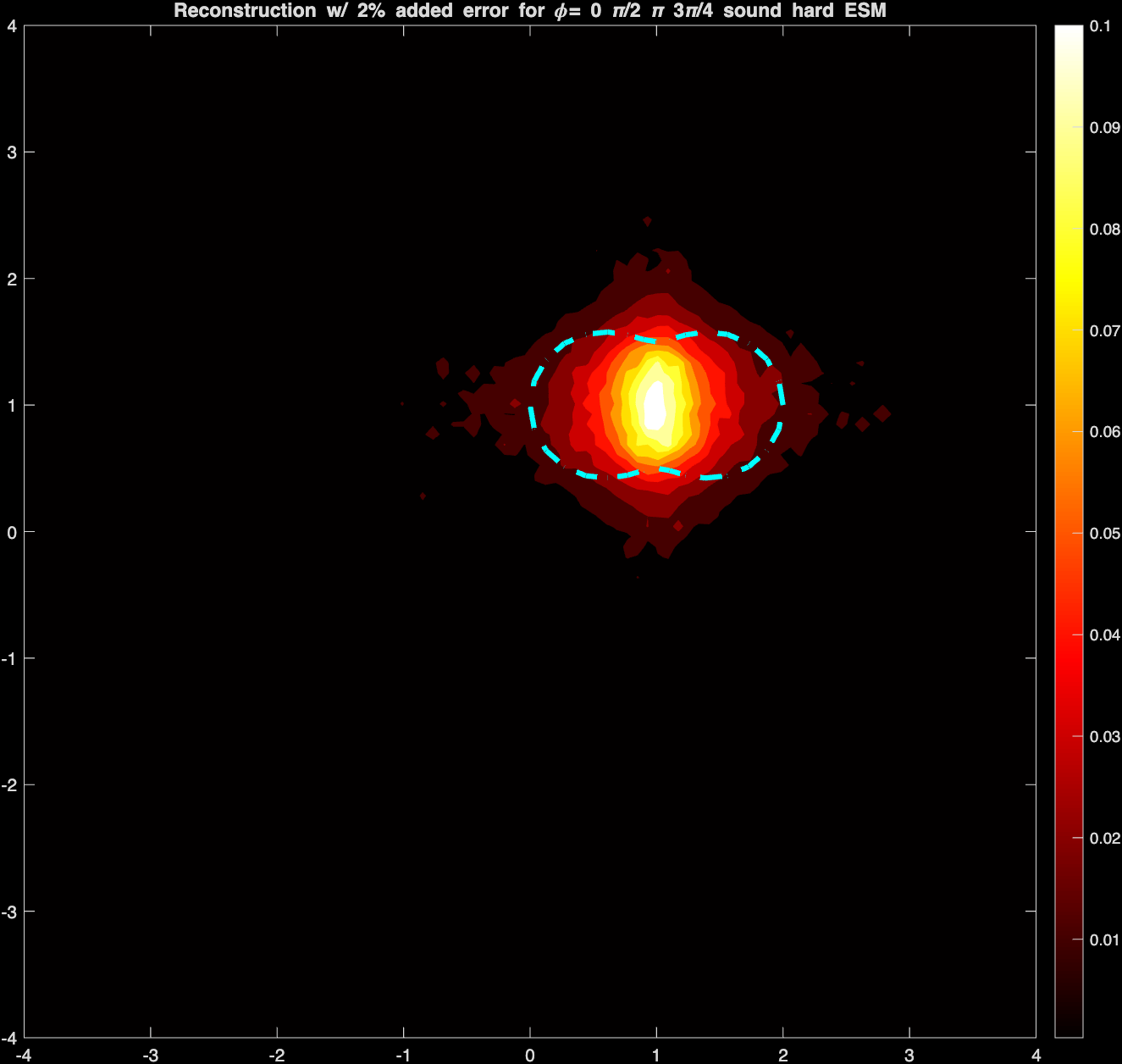}
        \caption{sound--hard reconstruction w/ $\mathbb S_{\text{inc},4}^1$}
    \end{subfigure}%
    \caption{Comparison of the reconstruction given by the sound--soft and sound--hard sampling disks for the peanut--shaped obstacle centered at $(1,1)$ with 2$\%$ added noise factorization--based ESM, with four incident directions.}\label{compare-peanut4inc}
\end{figure}

Now, we will provide another comparison of the two factorization--based ESMs for when one has less data. Here we will take the data given by $\mathbb S_{\text{inc},1}^1$, i.e., only one incident direction. We would expect this to give the worst reconstruction. As we see in Figure \ref{compare-star1inc}, even with this limited amount of data we see that one can recover the location and approximate size of the obstacle.  In this example, we take the wavenumber $\kappa=2$ and $\delta = 0.1$ to test the stability in the presence of significant noise added to the data. As in the previous two examples in this section, we only present the reconstructions given by the optimal radius which is, in both cases, $R=0.7812$. 

\begin{figure}[H]
    \centering
    \begin{subfigure}[t]{0.4\textwidth}
        \centering
        \includegraphics[scale=0.08]{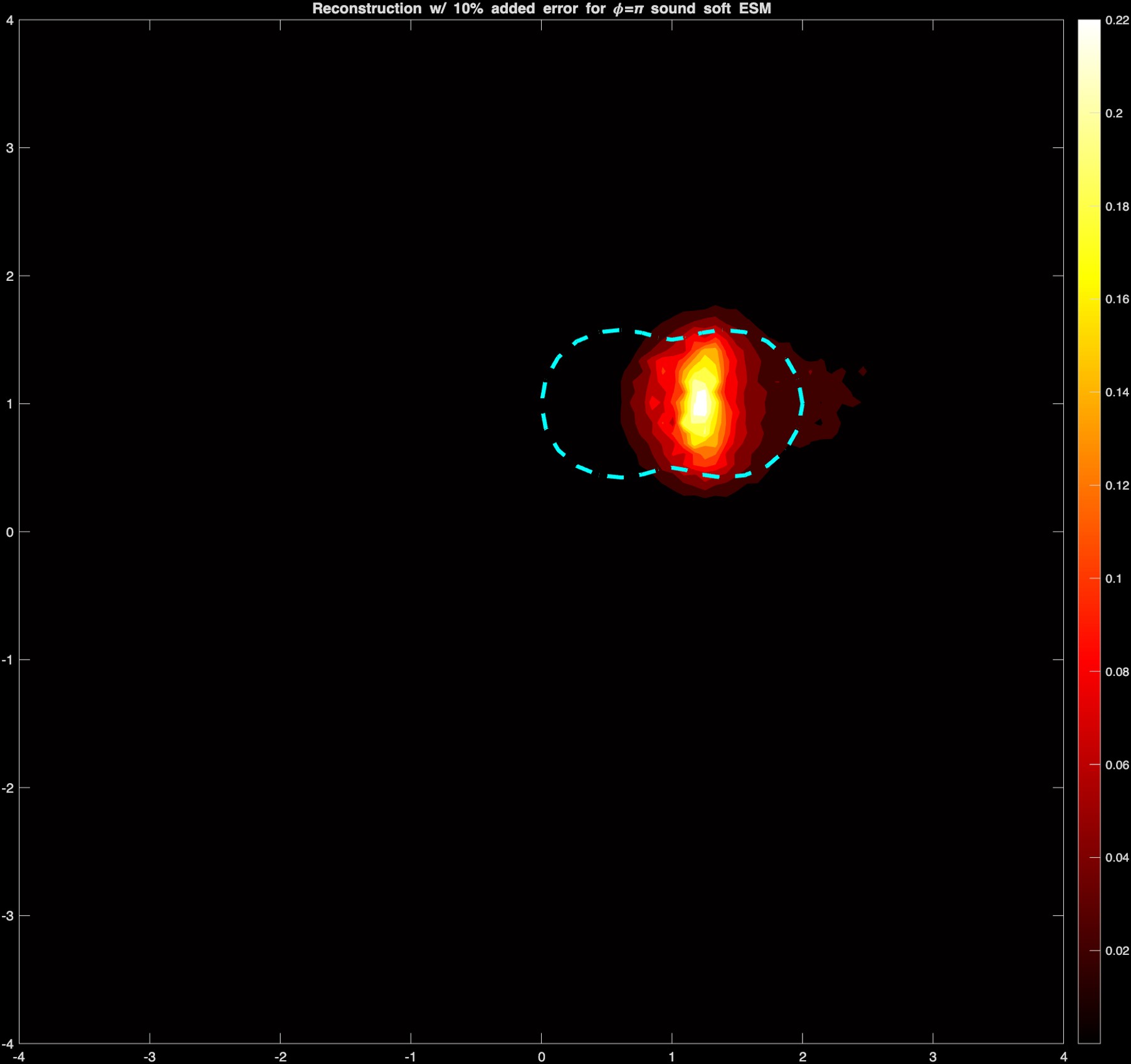}
        \caption{sound--soft reconstruction w/ $\mathbb S_{\text{inc},1}^1$}
    \end{subfigure}%
       ~ 
    \begin{subfigure}[t]{0.4\textwidth}
        \centering
        \includegraphics[scale=0.08]{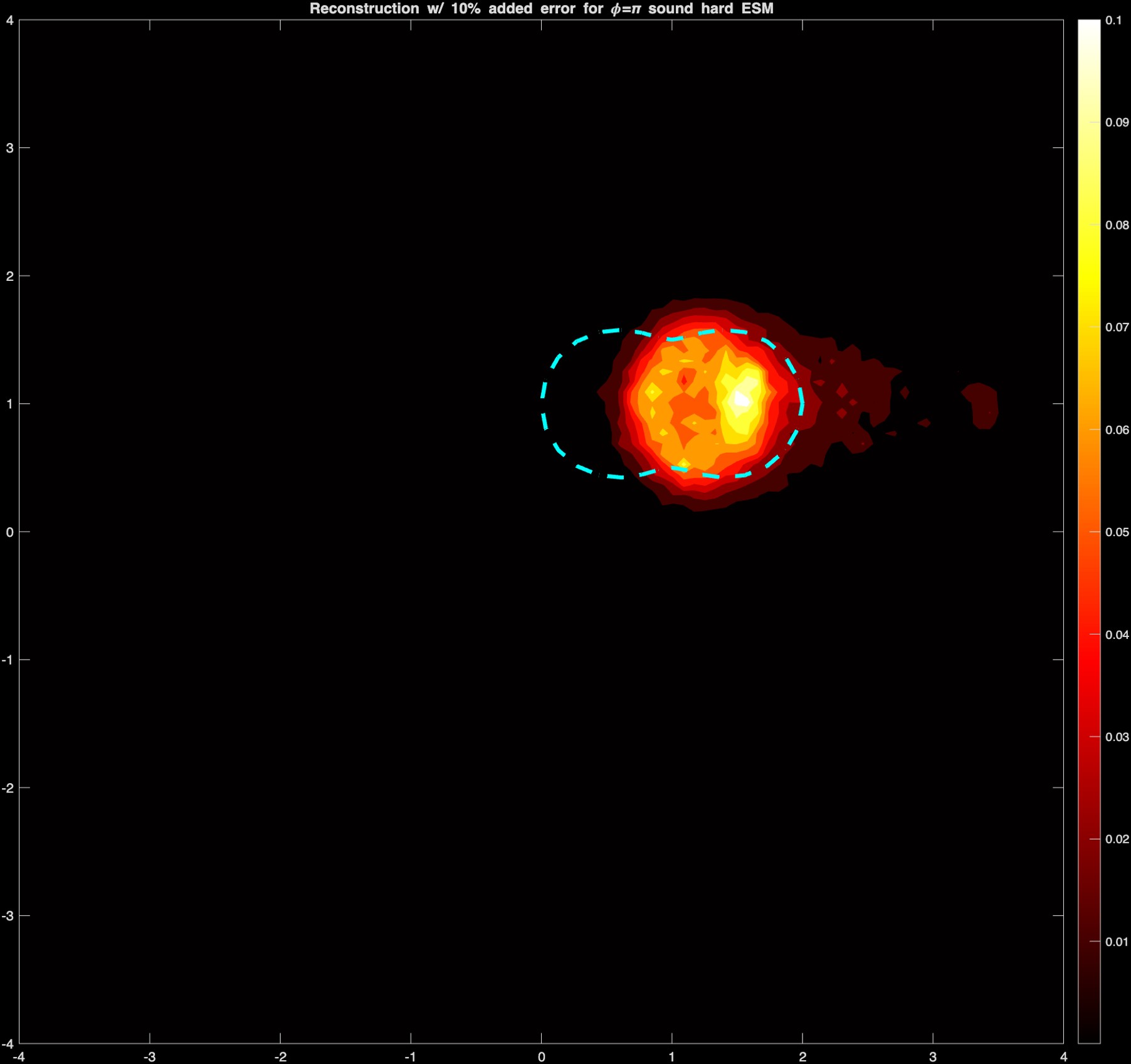}
        \caption{sound--hard reconstruction w/ $\mathbb S_{\text{inc},1}^1$}
    \end{subfigure}%
    \caption{Comparison of the reconstruction given by the sound--soft and sound--hard sampling disks for the peanut--shaped obstacle centered at $(1,1)$ with 10$\%$ added noise factorization--based ESM, with one incident direction.}\label{compare-star1inc}
\end{figure}

We can see that the reconstructions are comparable with the sound--hard sampling disk filling out more of the scatterer in this case in Figure \ref{compare-star1inc}. From our numerical testing, the  sound--hard sampling disk often gives a better reconstruction, especially with data from only one or two incident directions but the difference is minor and determined via visual inspection. We also want to remark that when one uses a sampling radius slightly larger than the optimal radius then the reconstruction does not suffer.

\subsection{Obstacle approximation via the `optimal' sampling disk}

Lastly, we consider finding an explicit formula for an approximating region to recover the scatterer from the support of the imaging functional in \eqref{multi_W_finite_noise}. Here, we will find an `optimal' sampling disk to approximate the clamped obstacle $D \approx B_{R^*}(z^*)$. To this end, we need to determine the optimal radius denoted $R^*$ and center $z^*$. Recall, that by Algorithm 1, we have a way to pick the optimal radius $R^* = R_{p^*}$ corresponds to the value of $p$ yielding the first reconstruction without any artifacts as seen in Figure \ref{findR-peanutSS}--\ref{findR-starSS}. From our numerical testing the optimal radius seems to satisfy that $R^* \approx 0.5\cdot \mathrm{diam}(D)$.

In order to pick the optimal center, we notice that in the above reconstructions, the imaging functional $W_{\mathrm{Recon}}(z)$ given in equation \eqref{multi_W_finite_noise} seems to always obtain its maximal value contained inside the scatterer. In particular, we see from Figures \ref{findR-starSH}--\ref{compare-peanut4inc} that the imaging functional is maximal at/near the geometric center of the object, especially in the case of data given by multiple incident directions. With this in mind, we will define the optimal center $z^*$ by the quantity 
$$ z^* = \operatorname*{arg\,max}_{z \in \mathcal{M}} W_{\mathrm{Recon}}(z) $$
which can be determined visually by using a data tip in the contour plot of the imaging functional. Note that this `optimal' sampling disk has a theoretical meaning and is not only justified by numerical experiments. Indeed, the imaging functional 
$$ W_{\mathrm{Recon}}(z) = 1/ \| \mathbf{g}^\alpha_z \|_2^2 \quad \text{is the regularized solution to} \quad \left(\mathbf{F}^{\sigma*}_z \mathbf{F}^\sigma_z \right)^{1/4} \mathbf{g}_z = u^\infty(\cdot \, ,d)$$
for the case of a single incident direction. Therefore, we see that the `optimal' sampling disk is the disk such that the above linear system has minimal norm. So this can be seen as the disk for which the measured scattering data matches the sound--soft or sound--hard disk with minimal energy. This is due to the fact that when $D \subseteq B_{R}(z),$ the above linear system is guaranteed to have a solution, implying that $\| \mathbf{g}^\alpha_z \|_2$ will be bounded. However, when $D \cap B_{R}(z) = \emptyset,$ there is no solution to the linear system, so $\| \mathbf{g}^\alpha_z \|_2$ will be unbounded.

Now, we provide approximations of the star and peanut--shaped obstacles using the `optimal' sampling disk. To this end, the `optimal' sampling disk is determined by the above discussion. In Figure \ref{reconw/ball-peanut1inc}, we provide the reconstruction of the peanut--shaped obstacle by the contour plot of the imaging functional using the sound--soft factorization--based ESM and the `optimal' sampling disk. In Figure \ref{reconw/ball-star2inc}, we provide the reconstruction of the star--shaped obstacle by the contour plot of the imaging functional using the sound--hard factorization--based ESM and the `optimal' sampling disk. For both examples we let $\delta=0.02$ and the wavenumber $\kappa=2$.

\begin{figure}[H]
    \centering
    \begin{subfigure}[t]{0.4\textwidth}
        \centering
        \includegraphics[scale=0.44]{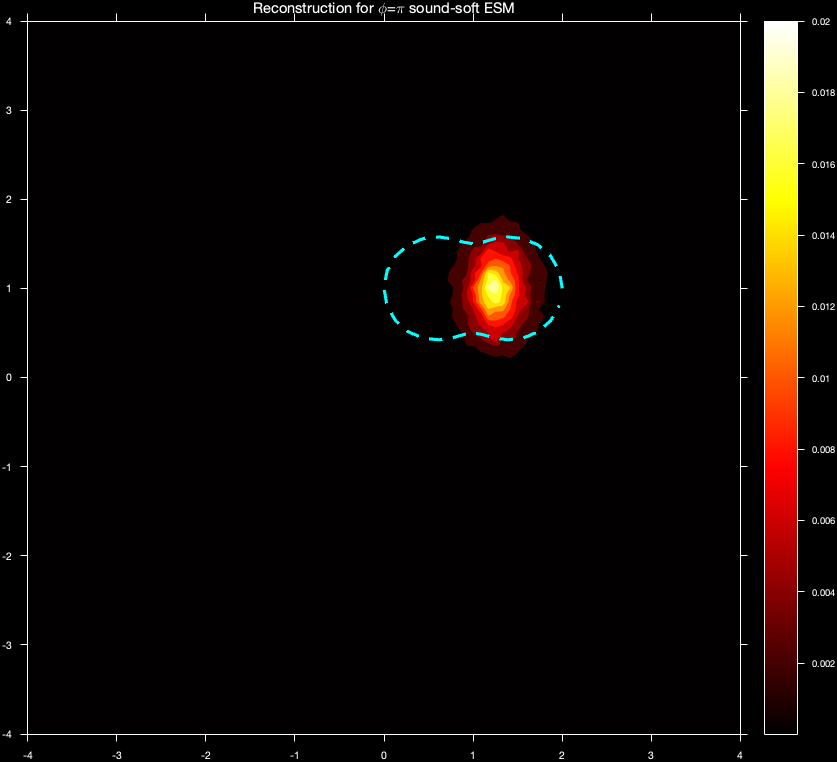}
        \caption{contour plot reconstruction}
    \end{subfigure}%
       ~ 
    \begin{subfigure}[t]{0.4\textwidth}
        \centering
        \includegraphics[scale=0.4]{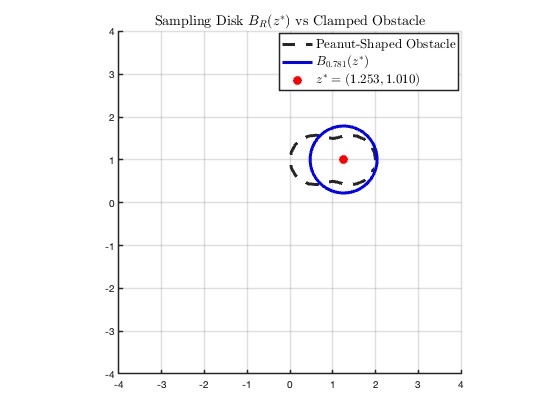}
        \caption{reconstruction via `optimal' sampling disk}
    \end{subfigure}%
    \caption{Reconstruction of a peanut--shaped obstacle with 2$\%$ added noise using the sound--soft (Dirichlet) factorization--based ESM and the `optimal' sampling disk, with one incident direction. The optimal radius is $R^* = 0.781$ and center $z^* = (1.253 , 1.010)$. }\label{reconw/ball-peanut1inc}
\end{figure}

\begin{figure}[H]
    \centering
    \begin{subfigure}[t]{0.4\textwidth}
        \centering
        \includegraphics[scale=0.44]{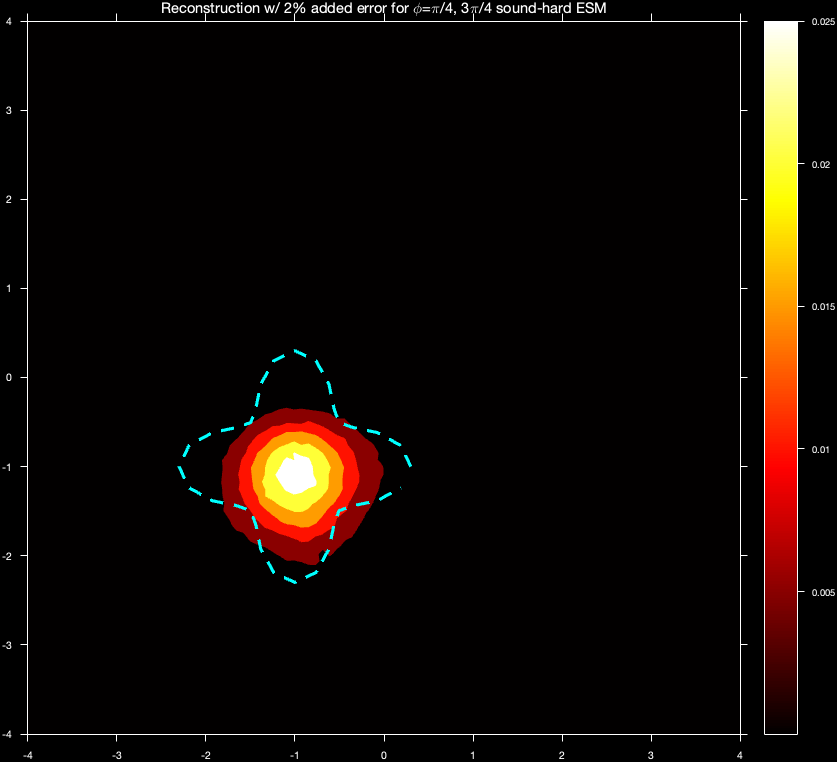}
        \caption{contour plot reconstruction}
    \end{subfigure}%
       ~ 
    \begin{subfigure}[t]{0.4\textwidth}
        \centering
        \includegraphics[scale=0.4]{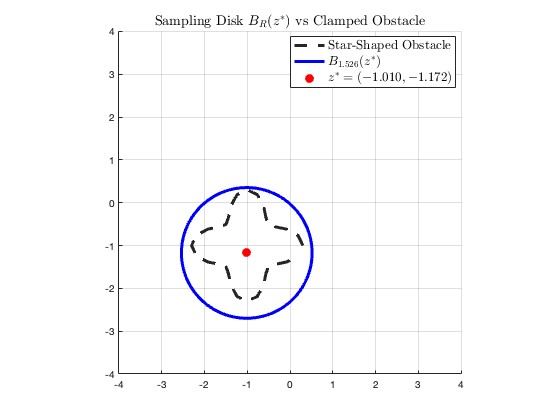}
        \caption{reconstruction via `optimal' sampling disk}
    \end{subfigure}%
    \caption{Reconstruction of a star--shaped obstacle with 2$\%$ added noise using the sound--hard (Neumann) factorization--based ESM and the `optimal' sampling disk, with two incident directions. The optimal radius is $R^* = 1.526$ and center $z^* = (-1.010 , -1.172)$. }\label{reconw/ball-star2inc}
\end{figure}

With this, we see that the `optimal' sampling disk gives a good approximation of the location and size of the obstacle. In particular, the combined encoding of location and size provides a useful initialization for more refined reconstruction techniques, including iterative schemes and optimization--based methods. Additionally, the disk--based approximation could serve as a prior or proposal mechanism in stochastic frameworks, such as Monte Carlo or Bayesian inversion approaches, where coarse geometric information can accelerate convergence and improve sampling efficiency as in \cite{LiDengSun2020,LiSunXu2020,HuangLiXu2026}.

\section{Conclusion}\label{end-FMbasedESM}
In this work we developed a novel implementation of the extended sampling method (ESM) for the inverse biharmonic (flexural) wave scattering problem by a clamped obstacle in two dimensions. The proposed reconstruction strategy is derived from the analytical framework of the factorization method, leading to computationally simple imaging functionals. Our central motivation is to exploit the ability of the ESM to localize an unknown scatterer using far--field data generated by only one (or a few) incident plane waves, but to derive an implementation based on the analytical framework of the factorization method similar to the work in \cite{MaHu2022}. 

We introduced two closely related ESM indicators based on using sound--soft (Dirichlet) and/or sound--hard (Neumann) sampling disks. The sound--soft formulation is the standard choice used in most of the ESM studies in the literature, while the sound--hard variant provides an additional mechanism for reconstructing the obstacle. Our numerical experiments demonstrate that both indicators can produce stable reconstructions in the presence of noise, and that the use of multiple incident directions improves robustness and geometric fidelity. Beyond just finding the location, we also developed a method for recovering an ``optimal'' sampling disk from the indicator data. The resulting disk approximation provides a reconstruction of the obstacle that encodes the region's approximate location and size.

Several directions for future work are of interest. First, it would be valuable to extend the factorization--based ESM developed here to other plate boundary conditions and more general physical parameters. Second, the current disk--based sampling operators introduce restrictions on the wavenumber (excluding Dirichlet/Neumann eigenvalues of the sampling disk). Therefore, it would be advantageous to investigate a sampling disk with impedance boundary conditions or penetrable sampling inclusions in order to mitigate or remove these constraints. Finally, studying limited--aperture configurations, three--dimensional analogues, and extensions to other wave models (e.g., elastic or electromagnetic settings) would further clarify the scope and practical utility of the factorization--based ESM methodology. \\

\noindent{\bf Acknowledgments:} The research of authors I. Harris and G. Ozochiawaeze is partially supported by the NSF DMS Grants 2509722 and 2208256.


\end{document}